\documentclass{amsart}
\usepackage{amsmath,amsfonts,amssymb,amsthm,amscd,comment,euscript}
\usepackage{mathbbol, mathtools, stmaryrd, bm}
\usepackage{color}
\usepackage{hyperref}
 \hypersetup{colorlinks=true, linkcolor=magenta, citecolor=violet}
\usepackage[all,arc]{xy}
\usepackage{graphicx}
\usepackage{ulem} 
\usepackage{tikz-cd, fancyhdr}
\allowdisplaybreaks

\newtheorem{thm}{Theorem}[section]
\newtheorem{cor}[thm]{Corollary}
\newtheorem{prop}[thm]{Proposition}
\newtheorem{lem}[thm]{Lemma}

\newtheorem{defn}[thm]{Definition}
\newtheorem{ex}[thm]{Example}

\newtheorem{rem}[thm]{Remark}

\usepackage{tikz}
\usepgflibrary{shapes.geometric}
\usetikzlibrary{arrows}

\tikzstyle{V}=[draw, fill =black, circle, inner sep=0pt, minimum size=1.5pt]
\tikzstyle{wV}=[draw, fill =white, circle, inner sep=0pt, minimum size=4.5pt]
\tikzstyle{bV}=[draw, fill =black, circle, inner sep=0pt, minimum size=4.5pt]
\tikzstyle{over}=[draw=white,double=black,line width=2pt, double distance=.5pt]

\newcommand\junk[1]{}

\newcommand{\nc}{\newcommand}
\nc{\rnc}{\renewcommand}
\nc{\bb}[1]{{\mathbb #1}}
\nc{\bbA}{\bb{A}}\nc{\bbB}{\bb{B}}\nc{\bbC}{\bb{C}}\nc{\bbD}{\bb{D}}
\nc{\bbE}{\bb{E}}\nc{\bbF}{\bb{F}}\nc{\bbG}{\bb{G}}\nc{\bbH}{\bb{H}}
\nc{\bbI}{\bb{I}}\nc{\bbJ}{\bb{J}}\nc{\bbK}{\bb{K}}\nc{\bbL}{\bb{L}}
\nc{\bbM}{\bb{M}}\nc{\bbN}{\bb{N}}\nc{\bbO}{\bb{O}}\nc{\bbP}{\bb{P}}
\nc{\bbQ}{\bb{Q}}\nc{\bbR}{\bb{R}}\nc{\bbS}{\bb{S}}\nc{\bbT}{\bb{T}}
\nc{\bbU}{\bb{U}}\nc{\bbV}{\bb{V}}\nc{\bbW}{\bb{W}}\nc{\bbX}{\bb{X}}
\nc{\bbY}{\bb{Y}}\nc{\bbZ}{\bb{Z}}
\nc{\mbf}[1]{{\mathbf #1}}
\nc{\bfA}{\mbf{A}}\nc{\bfB}{\mbf{B}}\nc{\bfC}{\mbf{C}}\nc{\bfD}{\mbf{D}}
\nc{\bfE}{\mbf{E}}\nc{\bfF}{\mbf{F}}\nc{\bfG}{\mbf{G}}\nc{\bfH}{\mbf{H}}
\nc{\bfI}{\mbf{I}}\nc{\bfJ}{\mbf{J}}\nc{\bfK}{\mbf{K}}\nc{\bfL}{\mbf{L}}
\nc{\bfM}{\mbf{M}}\nc{\bfN}{\mbf{N}}\nc{\bfO}{\mbf{O}}\nc{\bfP}{\mbf{P}}
\nc{\bfQ}{\mbf{Q}}\nc{\bfR}{\mbf{R}}\nc{\bfS}{\mbf{S}}\nc{\bfT}{\mbf{T}}
\nc{\bfU}{\mbf{U}}\nc{\bfV}{\mbf{V}}\nc{\bfW}{\mbf{W}}\nc{\bfX}{\mbf{X}}
\nc{\bfY}{\mbf{Y}}\nc{\bfZ}{\mbf{Z}}
\nc{\bfa}{\mbf{a}}\nc{\bfb}{\mbf{b}}\nc{\bfc}{\mbf{c}}\nc{\bfd}{\mbf{d}}
\nc{\bfe}{\mbf{e}}\nc{\bff}{\mbf{f}}\nc{\bfg}{\mbf{g}}\nc{\bfh}{\mbf{h}}
\nc{\bfi}{\mbf{i}}\nc{\bfj}{\mbf{j}}\nc{\bfk}{\mbf{k}}\nc{\bfl}{\mbf{l}}
\nc{\bfm}{\mbf{m}}\nc{\bfn}{\mbf{n}}\nc{\bfo}{\mbf{o}}\nc{\bfp}{\mbf{p}}
\nc{\bfq}{\mbf{q}}\nc{\bfr}{\mbf{r}}\nc{\bfs}{\mbf{s}}\nc{\bft}{\mbf{t}}
\nc{\bfu}{\mbf{u}}\nc{\bfv}{\mbf{v}}\nc{\bfw}{\mbf{w}}
\nc{\bfx}{\mbf{A}}
\nc{\bfy}{\mbf{B}}
\nc{\bfz}{\mbf{C}}
\nc{\bbx}{\mathbb{a}}
\nc{\bby}{\mathbb{b}}
\nc{\bbt}{\mathbb{t}}
\nc{\bbz}{\mathbb{c}}

\nc{\mcal}[1]{{\mathcal #1}}
\nc{\calA}{\mcal{A}}\nc{\calB}{\mcal{B}}\nc{\calC}{\mcal{C}}\nc{\calD}{\mcal{D}}
\nc{\calE}{\mcal{E}} \nc{\calF}{\mcal{F}}\nc{\calG}{\mcal{G}}\nc{\calH}{\mcal{H}}
\nc{\calI}{\mcal{I}}\nc{\calJ}{\mcal{J}}\nc{\calK}{\mcal{K}}\nc{\calL}{\mcal{L}}
\nc{\calM}{\mcal{M}}\nc{\calN}{\mcal{N}}\nc{\calO}{\mcal{O}}\nc{\calP}{\mcal{P}}
\nc{\calQ}{\mcal{Q}}\nc{\calR}{\mcal{R}}\nc{\calS}{\mcal{S}}\nc{\calT}{\mcal{T}}
\nc{\calU}{\mcal{U}}\nc{\calV}{\mcal{V}}\nc{\calW}{\mcal{W}}\nc{\calX}{\mcal{X}}
\nc{\calY}{\mcal{Y}}\nc{\calZ}{\mcal{Z}}
\nc{\fA}{\frak{A}}\nc{\fB}{\frak{B}}\nc{\fC}{\frak{C}} \nc{\fD}{\frak{D}}
\nc{\fE}{\frak{E}}\nc{\fF}{\frak{F}}\nc{\fG}{\frak{G}}\nc{\fH}{\frak{H}}
\nc{\fI}{\frak{I}}\nc{\fJ}{\frak{J}}\nc{\fK}{\frak{K}}\nc{\fL}{\frak{L}}
\nc{\fM}{\frak{M}}\nc{\fN}{\frak{N}}\nc{\fO}{\frak{O}}\nc{\fP}{\frak{P}}
\nc{\fQ}{\frak{Q}}\nc{\fR}{\frak{R}}\nc{\fS}{\frak{S}}\nc{\fT}{\frak{T}}
\nc{\fU}{\frak{U}}\nc{\fV}{\frak{V}}\nc{\fW}{\frak{W}}\nc{\fX}{\frak{X}}
\nc{\fY}{\frak{Y}}\nc{\fZ}{\frak{Z}}
\nc{\fa}{\frak{a}}\nc{\fb}{\frak{b}}\nc{\fc}{\frak{c}} \nc{\fd}{\frak{d}}
\nc{\fe}{\frak{e}}\nc{\fFf}{\frak{f}}\nc{\fg}{\frak{g}}\nc{\fh}{\frak{h}}
\nc{\fri}{\frak{i}}\nc{\fj}{\frak{j}}\nc{\fk}{\frak{k}}\nc{\fl}{\frak{l}}
\nc{\fm}{\frak{m}}\nc{\fn}{\frak{n}}\nc{\fo}{\frak{o}}\nc{\fp}{\frak{p}}
\nc{\fq}{\frak{q}}\nc{\fr}{\frak{r}}\nc{\fs}{\frak{s}}\nc{\ft}{\frak{t}}
\nc{\fu}{\frak{u}}\nc{\fv}{\frak{v}}\nc{\fw}{\frak{w}}\nc{\fx}{\frak{x}}
\nc{\fy}{\frak{y}}\nc{\fz}{\frak{z}}

\nc{\al}{\alpha}
\nc{\ep}{\epsilon}
\nc{\la}{\lambda}
\nc{\be}{\beta}
\nc{\ka}{\kappa}
\nc{\de}{\delta}
\nc{\na}{\nabla}
\nc{\ga}{\gamma}
\nc{\nap}{{\na_+}}
\nc{\nam}{{\na_-}}
\nc{\lefta}{\leftarrow}
\nc{\wt}{\widetilde}
\nc{\La}{\Lambda}
\nc{\tr}{\triangle}
\nc{\tw}{\tilde w}
\nc{\tW}{\tilde W}
\nc{\hP}{\hat P}
\nc{\hB}{\hat B}
\nc{\hY}{\widehat Y}
\nc{\hx}{\hat x}

\nc{\btau}{\boldsymbol{\tau}}

\DeclareMathOperator{\pt}{pt}

\DeclareMathOperator{\Hom}{Hom}

\DeclareMathOperator{\Sym}{Sym}
\DeclareMathOperator{\rev}{rev}
\DeclareMathOperator{\Frac}{Frac}

\DeclareMathOperator{\stab}{stab}

\newcommand{\unit}{\mathbf{1}}
\newcommand{\hh}{\mathbb{h}}

\title[Structure Constants]{Structure Constants in equivariant oriented cohomology of flag varieties}
\author{Rebecca Goldin, Changlong Zhong}

\address{George Mason University, Department of Mathematical Sciences, 4400 University Dr., Fairfax, VA 22030}
\email{rgoldin@gmu.edu}

\address{University at Albany, Department of Mathematics, CK399, 1400 Washington Ave, Albany, 12222}
\email{czhong@albany.edu}


\begin{document}

\begin{abstract}We introduce generalized Demazure operators for the equivariant oriented cohomology of the flag variety, which have specializations to various Demazure operators and Demazure-Lusztig operators in both equivariant cohomology and equivariant K-theory. In the context of the geometric basis of the equivariant oriented cohomology given by certain Bott-Samelson classes, we use these operators to obtain formulas for the structure constants arising in different bases. Specializing to divided difference operators and Demazure operators in singular cohomology and K-theory, we recover the formulas for structure constants of Schubert classes obtained in Goldin-Knutson \cite{GK19}. Two specific specializations result in formulas for the the structure constants  for
cohomological and K-theoretic stable bases as well; as a corollary we reproduce a formula for the structure constants of the Segre-Schwartz-MacPherson basis previously obtained by Su \cite{S19}. Our methods involve the study of the formal affine Demazure algebra, providing a purely algebraic proof of these results.
\end{abstract}

\maketitle

\section{Introduction}
Flag varieties  $G/B$ are among the most studied varieties in  topology and algebraic geometry. They have a cellular decomposition by Schubert cells, whose closures are called Schubert varieties. Schubert varieties are invariant under a torus action and, consequently, their torus-equivariant singular cohomology is spanned as a module by the Schubert classes. 

Other classes associated to Schubert varieties  in the equivariant singular cohomology 
$H^*_T$ and equivariant K-theory $K_T$  of the flag variety $G/B$
 include Chern-Schwartz-MacPherson  (CSM) classes and Motivic Chern (mC) classes, studied in \cite{AM15, AMSS17, RTV15, RV15, RTV17,  S15,  SZZ17}.
These classes coincide with the corresponding {\it stable bases} of Maulik-Okounkov \cite{MO19} for $H^*_T$ and $K_T$, of the Springer resolutions. 
 Due to this fact, we always refer to the CSM classes as the cohomological stable basis, and to the mC classes as the K-theoretic stable basis. These classes behave like 
 Schubert classes in their corresponding theories. Roughly speaking, Schubert classes in $H^*_T(G/B)$ and $K_T(G/B)$ are constructed by {\it Demazure operators} (also called {\it divided difference operators}), and  elements of the stable bases are constructed by {\it Demazure-Lusztig operators.} All these operators generate various Hecke-type algebras. 

Structure constants of Schubert classes are central objects in Schubert calculus, appearing in important questions of representation theory and combinatorics. In \cite{GK19}, the first author and Knutson obtain formulas for the structure constants in $H^*_T(G/B)$ and $K_T(G/B)$ using  geometric properties of Bott-Samelson resolutions of Schubert varieties. They pull-back the Schubert classes to the equivariant cohomology (or equivariant K-theory) of Bott-Samelson variety, apply the cup product in this variety, then push-forward back to $G/B$. In \cite{S19}, Su generalized this method to the so-called Segre-Schwartz-MacPherson (SSM) classes, a variant form of CSM classes.

We are interested in generalized cohomology theories, called  oriented cohomology theories, defined by Levine and Morel \cite{LM07}. These  cohomologies are contravariant functors defined on the category of smooth projective varieties over a field $k$ of characteristic 0 to the category of commutative rings, such that for proper maps, there is a push-forward map on cohomology groups. Examples include Chow rings (singular cohomology), K-theory and algebraic cobordism. Chern classes are defined for each oriented cohomology theory $\hh$, and there is an associated formal group law $F$ defined over $R=\hh(\pt)$.  The machinery works equivariantly  as well, resulting in a cohomology theory $\hh_T$ with an associated formal group law $F$ defined over $R = \hh_T(pt)$. 

For flag varieties, generalizing work of Kostant and Kumar \cite{KK86, KK90} on equivariant singular cohomology and equivariant K-theory of flag varieties, the ring $\hh_T(G/B)$ has a nice algebraic model, constructed in Hoffmann {\it et al.} in \cite{HMSZ14}, and studied in \cite{CZZ1, CZZ2, CZZ3} by Calm\`es, Zainoulline, and the second author. One can define the (formal) Demazure  operators $X_\al$ associated to each simple root $\al$. These operators generate a non-commutative algebra, called the formal affine Demazure algebra $\bfD_F$. 
 It is a free left $\hh_T(\pt)$-module with basis 
 $\{X_{I_w}\ \vert \  w\in W\}$, 
 where $X_{I_w}$ is, roughly speaking, a product of the operators $X_\al$,  with $I_w$ indicating a reduced word expression for $w$.

The algebra $\bfD_F$ is also a co-commutative co-algebra, where the coproduct comes from  the  twisted Leibniz rule of the operator $X_\al$. Taking the $\hh_T(\pt)$-dual, one obtains a commutative ring $\bfD^*_F$, 
 a free $\hh_T(\pt)$-module isomorphic to $\hh_T(G/B)$, together with a dual basis $\{X_{I_w}^*\ \vert \ w\in W\}$.
Indeed, for equivariant Chow group/singular cohomology/K-theory, $X_{I_w}^*$ coincides, up to various normalizations, to the Schubert 
 class associated with $w$. Then $H_T^*(G/B)$ and $K_T(G/B)$ are achieved with the same module basis, and a restricted coefficient ring: a polynomial ring for $H_T^*(G/B)$ and Laurent polynomial ring for $K_T(G/B)$.

We notice that  the product structure on $\bfD_F^*$  is obtained  by dualizing the coproduct structure of $\bfD_F$. 
It follows that the  structure constants of the basis $X_{I_w}^*$ may be deduced
from the twisted Leibniz rule of the product $X_{\be_1} X_{\be_2}\cdots X_{\be_k}$ for a reduced 
word $s_{\be_1}\cdots s_{\be_k}$
of $w\in W$. This is the main idea of the proof of  Theorem \ref{thm:Leibniz}, which implies the main result, Theorem \ref{thm:prod}. Specializing $\hh_T$ to equivariant singular cohomology and equivariant K-theory, we recover the formulas of the first author and Knutson in \cite{GK19}.

In the case of $H_T^*(G/B)$ and $K_T(G/B)$, replacing the Demazure operators $X_\al$ by the Demazure-Lusztig operators $T_\al$ and $\tau_\al^-$, one obtains the stable bases for $H^*_T(G/B)$ and $K_T(G/B)$, respectively. Both the cohomology stable basis and the K-theory stable basis  can be described in an analogous fashion to the story for Schubert classes. That is, the Demazure-Lusztig operators generate a
 degenerate affine Hecke algebra (for equivariant cohomology) 
 and an  
 affine Hecke algebra (for equivariant K-theory).  The 
 dual elements to products of these operators are essentially the cohomological/K-theoretic stable bases, so  their
  respective twisted Leibniz rules
result in a formula for
 the structure constants of stable bases. For instance, for cohomology, we recover the formula of Su \cite{S19} (see Remark \ref{rem:Su}).

To work with the Demazure operators $X_\al$ and Demazure-Lusztig operators  $T_\al$ at the same time, we define a general operator $Z_\al$ (see \textsection 3) in a ring containing $\bfD_F$, which can be specialized to $X_\al$ and $T_\al$. Our main results are Theorems \ref{thm:prod} and \ref{thm:cohstab}, which 
state a 
formula for structure constants of the basis determined by $Z_\al$ and apply it to the cohomological stable basis.

The paper is organized as follows: In \textsection 2 we recall necessary notation introduced by the second author in \cite{CZZ1, CZZ2, CZZ3}. We recall the definition of  a Demazure element,  the  formal affine Demazure algebra, its dual, and relation with $\hh_T(G/B)$. In \textsection 3 we prove the twisted Leibniz rule for the operator $Z_\alpha$, which is used to derive the structure constants of the basis $Z_{I_w}^*$ in \textsection 4. In \textsection 5, we specialize our result to Demazure operators in singular cohomology and K-theory, and recover the formulas in \cite{GK19}.
In \textsection 6 we specialize our result to  Demazure-Lusztig operators in singular cohomology, which, as a by-product, recovers the formula due to Su in \cite{S19}.
In \textsection 7 we consider Demazure-Lusztig operators in K-theory and obtain a formula for the structure constants of the K-theoretic stable basis.
In \textsection 8, for equivariant oriented cohomology, we generalize some results of Kostant-Kumar  (\cite[Proposition 4.32]{KK86}, \cite[Lemma 2.25]{KK90}) by relating our formula for structure constants with  a restriction formula of Schubert classes. 

\noindent{\it Acknowledgments}: 
The first author was partially supported by National Science Foundation grant DMS-2152312.

\section{Preliminary} \label{sec:prelim}
We follow notation used in \cite{CZZ1,CZZ2, CZZ3}. Let $\Sigma\hookrightarrow \Lambda^\vee, \al\mapsto \al^\vee$ be a semi-simple root datum of rank $n$. That is, $\Sigma$ is the finite set of roots,  $\La$ is the lattice and $\La^\vee$ is its dual. Let $\{\al_1,...,\al_n\}$ be the set of simple roots, $\Sigma^+$ and $\Sigma^-$ be the set of positive and negative roots, respectively. 

Let $W$ be the Weyl group generated by the associated simple reflections $s_i:=s_{\al_i}$. Denote by $\le$ the Bruhat order, and let $\ell(v)$ be the length of an element $v\in W$.  Note that $W$ acts on $\La$ since it preserves the root system. For each sequence $I=(i_1,...,i_k)$ with $i_j\in [n]$, denote the product $s_{i_1}\cdots s_{i_k}\in W$ by 
$\prod I,$ in which we keep track both of the concatenated sequence of simple reflections and the resulting element of $W$.
If $\prod I$ is a reduced word expression for the resulting Weyl group element, we say that $I$ is a {\it reduced sequence}. Following  \cite[\textsection 1]{GK19}, define the {\it Demazure product} 
$$\wt \prod I=s_{i_1}\cdots s_{i_k}$$ subject to the braid relations and $s_i^2=s_i$ for all $i$. Observe that $\prod I=\wt \prod I$  when $I$ is a reduced sequence. 
When $I$ is a reduced sequence for $w$, we may denote it by $I_w$ and abuse notation by calling it a reduced word for $w$. Finally, let $I^{\rev}$  denote the sequence obtained by reversing the sequence $I$. 

Let $F$ be a formal group law  over the coefficient ring $R$. Examples of formal group laws include the additive formal group law $F_a=x+y$ and the multiplicative formal group law $F_m=x+y-xy$. Suppose the root datum together with the formal group law satisfy the regularity condition of \cite[Lemma 2.7]{CZZ3}. This guarantees  that all the properties that we use from \cite{CZZ1, CZZ2, CZZ3}  hold. Indeed, the regularity condition guarantees that the elements  $x_\alpha, \al\in \La$ defined in $S$ below  
are non-zero-divisors. In particular, the Demazure operators $X_\alpha$ for simple roots $\alpha$ are well defined.

 Let $G$ be a split semi-simple  linear algebraic group with maximal torus $T$ and a Borel subgroup $B$. Let the associated root datum of $G$ be $\Sigma\hookrightarrow \La^\vee$, so $\La$ is the group of characters of $T$. 

Let $\hh$ be an oriented cohomology theory of Levine and Morel. Roughly speaking, it is a contravariant functor from the category of smooth projective varieties to the category of commutative rings such that there is a push-forward map for any proper map. The Chern classes of vector bundles are defined. Associated to $\hh$, there is a formal group law  is $F$ defined over $R=\hh(\pt)$. That is, the first Chern class of line bundles over a smooth projective variety $X$ satisfies 
$$
c^\hh_1(\calL_1\otimes \calL_2)=F(c_1^\hh(\calL_1), c_1^\hh(\calL_2)).
$$ For example, $F_a$ (resp. $F_m$) is associated to the Chow group (or singular cohomology) (resp. K-theory). Both can be extended to the torus equivariant setting. We assume the equivariant cohomology theory $\hh_T$ is Chern-complete over the point for $T$, that is, the ring $\hh_T(\pt)$ is separated and complete with respect to the topology induced by the $\gamma$-filtration \cite[Definition 2.2]{CZZ3}. In particular, this includes the completed equivariant Chow ring, the completed equivariant K-theory and equivariant algebraic cobordism.

Let $S$ be the formal group algebra defined in \cite{CPZ13}:
\begin{equation}\label{eq:S}
S= R[[\La]]_F:=R[[x_\la \vert \la\in \La]]/J_F,
\end{equation}
where $J_F$ is the closure of the ideal generated by $x_0$ and $x_{\la+\mu}-F(x_\la, x_\mu),$ for all  $\la, \mu\in \La$. Indeed, if $\{t_1,...,t_n\}$ is a basis of $\La$, then $S$ is (non-canonically) isomorphic to $R[[x_{t_1},...,x_{t_n}]]$. According to \cite[\textsection 3]{CZZ3},  $S\cong \hh_T(\pt)$ with $x_\la$ corresponding to $c^\hh_1(\calL_\la)$ where  $\calL_\la$ is the line bundle  associated to $\la\in \La$. Since $x_{-\la}$ is the formal inverse of $x_\la$, i.e.  $F(x_\la, x_{-\la})=0$ in $S$, we may write 
\[
x_{-\la}=-x_{\la}+\text{higher~degree~terms}\in S. 
\]
Define $Q:=S[\frac{1}{x_\al} \vert \al\in \Sigma]$. We will frequently need the special element of $Q$ 
 given by $\kappa_\la:=\frac{1}{x_\la}+\frac{1}{x_{-\la}}$. Note that $\kappa_\la$ actually belongs to $S$. Note also that the action of $W$ on $\La$ 
induces an action of $W$ on $S$. 

\begin{ex}\label{ex:FaFm} Two cases of the formal product appear widely in the literature \cite[\textsection 2]{CPZ13}. 
\begin{enumerate}
\item  If $F=F_a$ with $R=\bbZ$, then $\hh$ is the singular cohomology/Chow groups, and  $S\cong \Sym_{\bbZ}(\La)^\wedge$  ( with $x_\la\mapsto \la$) is the completion of the polynomial ring at the augmentation ideal. In this case $x_{-\la}=-x_\la$ and $\kappa_\la=0$.
\item If $F=F_m$ with $R=\bbZ$, then $\hh$ is K-theory, and $S\cong \bbZ[\La]^\wedge$ (with ${x_\la\mapsto 1-e^{-\la}}$) is the completion of the Laurent polynomial ring at the augmentation ideal. In this case $x_{-\la}=\frac{x_\la}{x_\la-1}$, and $\kappa_\la=1$. 
\end{enumerate}
\end{ex}

To obtain equivariant cohomology $H_T^*(X)$  and equivariant K-theory $K_T(X)$, we restrict the coefficient ring to  $S^a=Sym[\Lambda]$ and  $S^m=\mathbb Z[\Lambda]$, respectively.

\subsection{The operator algebras $Q_W$ and $\bfD_F$} 
This paper is concerned with various divided difference operators acting on $\hh_T(G/B)$, the equivariant oriented cohomology of $G/B$. To create an algebraic framework for these operators, following \cite{CZZ1, CZZ2} we 
localize $S$ at $\{x_\alpha\}$ to create an algebra out of this localization and the Weyl group, as follows.

Let $S$ be the ring described in \eqref{eq:S}, and let $Q:=S[\frac{1}{x_\al} \vert \al\in \Sigma]$.
Define
$Q_W:=Q\rtimes R[W],$
as a left $Q$-module with basis $\{\de_w\}, w\in W$. 

We shall see that $Q_W$ acts on its dual space $Q^*_W$, which is identified with $Q\otimes_S \hh_T(G/B)$, the cohomlogy of $G/B$ with inverted Chern classes.

We impose a product on $Q_W$ by
\[
{(}p\de_w {)}{(}p'\de_{w'}{)}=pw(p')\de_{ww'}, \quad \mbox{for all }p,p'\in Q,\mbox{ and } w,w'\in W,
\]
using the natural $W$ action on $Q$ induced from that on $\La$ and extending linearly. Note that $Q$ identified with $Q\delta_e$ is a subring of $Q_W$ under this product, where $e\in W$ denotes the identity element of $W$.
We routinely abuse notation and write $\de_\al$ for $\de_{s_\al}$, and use  $1=\de_e$ to denote the identity element of $Q_W$.
The ring $Q_W$ acts on $Q$ by 
\[p\de_w\cdot p'=pw(p'),\quad \mbox{for all } p, p'\in Q.\]

The action of $Q_W$ on $Q$ induces a coproduct structure on $Q_W$ as follows. Let $\eta=\sum\limits_{w\in W} q_w\de_w\in Q_W$. Then 
$$
\eta\cdot (pq) = \sum_w q_w w(pq) = \sum_w q_w w(p) w(q) = \sum_w q_w(\de_w\cdot p)(\de_w\cdot q).
$$
This action factors through the coproduct $\Delta: Q_W \rightarrow Q_W\otimes_Q Q_W$
\begin{equation}\label{eq:coprod}
\Delta(\eta) = \sum_w q_w\Delta(\de_w) = \sum_w q_w \de_w \otimes \de_w.
\end{equation}
In other words, the coproduct structure on $Q_W$ is induced from the $Q_W$-action on $Q$.

For any simple root $\al$
we, define the Demazure element $X_\al$ 
and the push-pull element $Y_\al$ in $Q_W$:
\begin{align*}
X_{\al}=\frac{1}{x_{\al}}(1-\de_{\al}) \qquad\mbox{and}\qquad Y_{\al}=\frac{1}{x_{-\al}}+\frac{1}{x_{\al}}\de_{\al}. 
\end{align*}
We observe the relationship $Y_{\al}=\ka_{\al}-X_{\al}$. In particular, 
 if $F=F_a$ (resp. $F=F_m$), then $Y_\al=-X_\al$ (resp. $Y_\al=1-X_\al$). 

The way $Q_W$ acts on $Q$ implies that $X_\al$ acts in a fashion similar to the Demazure operator defined in \cite{Dem74} (and there denoted $D_\al$). In particular, ${X_\al\cdot S \subset S}$ and, for any $r\in R, X_\al\cdot r=0$ and $\de_\al\cdot r = s_\al(r)=r$.

Let $\bfD_F$ be the $R$-subalgebra of $Q_W$ 
$$
\bfD_F=\langle S, X_{\al_1},\dots, X_{\al_n}\rangle
$$
generated by $S$ and the elements $X_\al \in Q_W$ for simple roots $\al$, and call it the {\it formal affine Demazure algebra.} It is also generated by $S$ and $\{Y_\al:  \al ~\mbox{simple}\}$. As a left $S$ module,  $\bfD_F$ is a also free with basis $\{X_{I_w}\}_{w\in W}$, or with basis $\{Y_{I_w}\}_{w\in W}$; see \cite[Proposition 7.7]{CZZ1}.

Let  $w =s_{i_1}\cdots s_{i_k}$ be a reduced word decomposition and $I_w = (i_1,...,i_k)$ the corresponding sequence of reflections. Define
\begin{align*}
X_{I_w} = X_{\al_{i_1}}\cdots X_{\al_{i_k}} \qquad\mbox{and} \qquad 
Y_{I_w} = Y_{\al_{i_1}}\cdots Y_{\al_{i_k}}.
\end{align*}
In particular, $X_{(i)}=X_{\al_i}$ and $Y_{(i)}=Y_{\al_i}$, though we eliminate parentheses when there is no confusion. We write $X_e:=1\in Q_W$ to indicate $X_I$ when $I$ is the empty sequence.

The Demazure and push-pull elements have the following properties:
\begin{lem}\cite[Proposition 3.2]{Z15}\label{lem:1} Let $\al$ and $\be$ be simple roots. 
The following identities hold in $Q_W$:
\begin{enumerate}
\item \label{item:1}$X_\al^2=\ka_\al X_\al, \quad Y_\al^2=\ka_\al Y_\al$.
\item \label{item:2}$X_\al p=s_\al(p)X_\al+X_\al\cdot p, \quad p\in Q$. 
\item \label{item:3}If $(s_\al s_\be)^2=e$, then $X_{\al}X_\be=X_\be X_\al$. 
\item \label{item:4} If $(s_\al s_\be)^3=e$, then $X_\be X_\al X_\be-X_\al X_\be X_\al =\ka_{\al \be }X_\al -\ka_{\be \al }X_\be $, where 
\[
\ka_{\al \be }=\frac{1}{x_{\al +\be }x_\be }-\frac{1}{x_{\al +\be }x_{-\al }}-\frac{1}{x_\al x_\be }. 
\]
Furthermore, $\ka_{\al \be }\in S$ by \cite[Lemma 6.7]{HMSZ14}. 
\item \label{item:5} Suppose $s_\al s_\be $ has order $m$ with $m=4$ or $6$,  and $I_w$ is a choice of reduced word for $w\in W$.  Then \[{\underbrace{X_\al X_\be X_\al \cdots}_m }-{\underbrace{X_\be X_\al X_\be \cdots}_m}=\sum_{v\in W}c_{I_v}X_{I_v},\] where $c_{I_v}=0$ if $v\not\leq\underbrace{s_\al s_\be s_\al \cdots}_m$. Moreover, $c_{I_v}=0$ if $\ell(v)=m-1$ or $v=e$. 
\end{enumerate}
\end{lem}
Lemma \ref{lem:1}.\eqref{item:4}-\eqref{item:5} imply that
 the operators $X_\al$ (and similarly $Y_\al$) do not satisfy braid relations for general $F$. For $F=F_a$ or $F=F_m$, they do; in these cases, the coefficients $\ka_{\al \be}$ and $c_{I_v}$ all vanish. In general, $X_{I_w}$ and $Y_{I_w}$ depend on the choice of $I_w$ due to this failure of braid relations.

For the purposes of this paper, we fix a reduced sequence $I_w$ of $w$ for each $w\in W$. While the specific coefficients and calculations regarding $X_{I_w}$ and $Y_{I_w}$  depend on this choice, statements regarding bases and ring phenomena do not.

By construction, $\{\de_v:\ v\in W\}$ form a basis of $Q_W$ as a module over $Q$. In \cite{CZZ1}, and extended in \cite{CZZ2}, the second author proves that $\{X_{I_v}: \ v\in W\}$  and  $\{Y_{I_v}: \ v\in W\}$  also form bases of $Q_W$ as a module over $Q$, and that the change of basis matrix from $\{X_{I_v}\}$ (or from $\{Y_{I_v}\}$) to $\{\de_v\}$ consists of elements of $S$.  In particular, $\{\delta_v\}$ are elements of $\bfD_F$. The lower-triangularity of the change of bases matrices is expressed in the following lemma.

\begin{lem}\cite[Lemma 3.2, Lemma 3.3]{CZZ2} \label{coefficientsforinverting}
For each $v\in W$, choose a reduced decomposition of $v$ and let $I_v$ be its corresponding sequence. There exist elements $a^X_{I_w,v}\in Q$ for $v\in W$,  and $b^X_{w,I_v}\in S$ 
 such that
$$
X_{I_w}=\sum_{v\le w}a^X_{I_w,v}\ \de_v, \mbox{ and}\quad \de_{w}=\sum_{v\le w}b^X_{w,I_v}X_{I_v}.
$$
Similarly, there exist $a^Y_{I_w,v}\in Q$ and $b^Y_{w,I_v}\in S$ such that
$$
Y_{I_w}=\sum_{v\le w}a^Y_{I_w,v}\ \de_v, \mbox{ and}\quad \de_{w}=\sum_{v\le w}b^Y_{w,I_v}Y_{I_v}.
$$
\end{lem}
Notice that nonzero coefficients $b^X_{w, I_v}$ are elements of $S$ with $v\leq w$.

\begin{ex}\label{ex:changebasis}
Consider the root datum $A_2$, with 
$$W=\{e, s_1, s_2, s_1s_2, s_2s_1, w_0\},$$ where $w_0$ is the longest element and $s_i$ is the reflection corresponding to $\alpha_i$ for $i=1,2$. We fix the reduced sequence $I_{w_0} = (1,2,1)$ for $w_0$.
For simplicity,  let $\al_{13}=\al_1+\al_2$. By direct computation,
\begin{align*}
&\de_{e\phantom{1}} =X_e &\qquad &\de_{s_1s_2}=1-x_1X_{(1)}-x_{\al_{13}}X_{(2)}+x_{\al_1}x_{\al_{13}}X_{(1,2)},\\
&\de_{1} =1-x_{\al_1}X_{(1)}&\qquad &\de_{s_2s_1}=1-x_{\al_2}X_{(2)}-x_{\al_{13}}X_{(1)}+x_2x_{\al_{13}}X_{(2,1)},\\ 
&\de_{2}  =1-x_{\al_2}X_{(2)} &\qquad  &\de_{w_0\phantom{w}}=1-x_{\al_{13}}X_{(2)}-(x_{\al_1}+x_{\al_2}-\kappa_{\al_1} x_{\al_1}x_{\al_2})X_{(1)}  \\
\phantom{\de_2} &\phantom{=1-x_2X_2 }&\quad &
+x_{\al_1}x_{\al_{13}}X_{(1,2)}+x_{\al_2}x_{\al_{13}}X_{(2,1)}-x_{\al_1}x_{\al_2}x_{\al_{13}}X_{I_{w_0}}.\\
\end{align*}
\end{ex}

\subsection{The dual operator algebras}\label{se:dualopalg}

The dual $Q$-module
$$
Q_W^*=\Hom_Q(Q_W, Q)\cong \Hom(W, Q),
$$ 
contains a natural basis $\{f_w\}_{w\in W}$ dual to $\{\delta_w\}_{w\in W}$, defined by 
$$
\langle f_w, \de_v\rangle=\begin{cases}
					1&\mbox{if $w=v$};\\
					0& \mbox{otherwise}.
					\end{cases}
$$

One may think of $Q_W^*$ as the $T$-equivariant oriented cohomology of $W$ with the trivial $T$ action, tensored with $Q$.
In particular, 
$$
Q_W^* = Q\otimes_S \hh_T(W) =  Q\otimes_S \hh_T(G/B).
$$
The module $Q_W^*$ forms a ring with product $f_w f_v = 1$ if an only if $w=v$, and $0$ otherwise, extended linearly to all elements of $Q_W^*$, 
and unity $\unit=\sum_{w\in W}f_w$.  This product structure is equivalent to the one induced from the 
coproduct structure (see \textsection\ref{sec:prod} below).

The ring $Q_W$ acts on $Q_W^*$ by 
\[
\langle z\bullet f, z'\rangle =\langle f, z' z\rangle , \quad \mbox{for all }z, z'\in Q_W,\  f\in Q_W^*.
\]
In the bases $\{\de_w\}$ of $Q_W$ and $\{f_w\}$ of $Q_W^*$, the action has explicit formulation
\begin{equation}p\de_w\bullet (qf_v)=qvw^{-1}(p)f_{vw^{-1}}, \mbox{ for all }p,q\in Q. 
\end{equation}
Denote 
$$
\quad \pt_w=(\prod_{\al<0}x_\al)\bullet f_w=w\left(\prod_{\al<0}x_\al\right)f_w\in Q_W^*. 
$$
 
Let  $\bfD_F^*:=\Hom_S(\bfD_F, S)\subset Q_W^*$ be the dual $S$-module to $\bfD_F$.  It is proved in \cite[Lemma 10.3]{CZZ2} that $\pt_w\in \bfD_F^*$. 
Let
\begin{align*}
\zeta_{I_w}^X &=X_{I_w^{\rev}}\bullet \pt_e, \mbox{and}\\
\zeta_{I_w}^Y& =Y_{I_w^{\rev}}\bullet \pt_e.
\end{align*} 
Then $\{\zeta_{I_w}^X\}$ forms a basis of $D_F^*$ over $S$, as does $\{\zeta_{I_w}^Y\}$.

Finally, let $\{X_{I_w}^*\},$ (respectively $\{Y_{I_w}^*\}$) be the bases dual to $\{X_{I_w}\}$ (resp. $\{Y_{I_w}\}$) in $\bfD^*_F$, which are also $Q$-basis of $Q_W^*$. 

 The classes $X_{I_v}^*$ for each $v\in W$ are determined by duality.  Under the dual pairing, 
\begin{align*}
\langle X_{I_v}^*, \de_w \rangle &= \langle X_{I_v}^*,\sum_{u\in W} b_{w,I_u}^X X_{I_u} \rangle 
= b_{w, I_v}^X.
\end{align*}
Set
$
X_{I_v}^* = \sum\limits_{u\in W} m_{I_v,u} f_u,
$
which implies
\begin{align*}
\langle X_{I_v}^*, \de_w \rangle &=\langle \sum\limits_{u\in W} m_{I_v,u} f_u, \de_w\rangle = m_{I_v,w}, 
\end{align*}
and thus $X_{I_v}^*=   \sum\limits_{w\in W} b_{w, I_v}^X f_w$. 
\begin{ex}\label{ex:dualbasischange}
Consider the root datum $A_2$, with 
$W=\{e, s_1, s_2, s_1s_2, s_2s_1, w_0\}.$
Fix the reduced sequence $w_0=s_1s_2s_1$. The calculations from Example~\ref{ex:changebasis} imply
\begin{align*}
X_e^*&=\unit=\sum_{w\in W}f_w ,& &X_{(1,2)}^*=x_{\al_1}x_{\al_{13}}(f_{s_1s_2}+f_{w_0})\\
X_{(1)}^*&=-x_{\al_1}(f_{s_1}+f_{s_1s_2})-x_{\al_{13}}f_{s_2s_1}-yf_{w_0} ,& &X_{(2,1)}^*=x_{\al_2}x_{\al_{13}}(f_{s_2s_1}+f_{w_0})\\
X_{(2)}^*&=-x_{\al_2}(f_{s_2}+f_{s_2s_1})-x_{\al_{13}}(f_{s_1s_2}+f_{w_0}),& &X^*_{I_{w_0}}=-x_{\al_1}x_{\al_2}x_{{\al_{13}}}f_{w_0},
\end{align*}
where $y=x_{\al_1}+x_{\al_2}-\kappa_{\al_1} x_{\al_1}x_{\al_2}$. In case $F=F_a$ or $F_m$, we have $y=x_{\al_{13}}$.
\end{ex}
The following proposition explains the relationship between the algebraic construction above and equivariant oriented cohomology of $G/B$.

For each reduced sequence $I_w$, let $\calX_{I_w}\to G/B$ denote the Bott-Samelson resolution. The push-forward in $\hh_T$ of the fundamental class along this resolution is called the {\it Bott-Samelson class of $I_w$}, which we denote by $\eta_{I_w}$.
Define a map
$$
\Phi: \bfD^*_{F}\quad \longrightarrow \quad \hh_T(G/B)
$$
given by $\Phi(\zeta_{I_w}^Y) = \eta_{I_w}$ and $\Phi({\unit}) = [G/B]$, the fundamental class of $G/B$, and extended as a module over $S$.

\begin{prop}\label{prop:Schubert} 
\begin{enumerate} The isomorphism $\Phi$ satisfies the following properties:
\item \cite[Theorem 8.2, Lemma 8.8]{CZZ3} The map $\Phi$ is a functorial isomorphism.
\item \cite[Theorem 14.7]{CZZ2} 
 The basis $\{\Phi(X_{I_w}^*):\ w\in W\}$ (resp. $\{\Phi(Y_{I_w}^*)):\ w\in W\}$) is dual to $\Phi(\zeta^X_{I_w})$ (resp. $\Phi(\zeta^Y_{I_w})$) via the nondegenerate dual pairing on $\hh_T(G/B)$ given by multiplying and pushing forward to a point.
 
\item \cite[Corollary 6.4]{CZZ3}
Let $i_w: wB \hookrightarrow G/B$ be the inclusion of the $T$-fixed point corresponding to $w\in W$, and $(i_w)_*: \hh_T(wB)\rightarrow \hh_T(G/B)$ be the pushforward map. Then  $\Phi(\pt_w) = (i_w)_*(1)$.

\item There is a commutative diagram
\[
\xymatrix{\bfD_F^*\ar@{^(->}[rr]\ar[d]^{ \cong} &  & Q_W^*\ar[d]^{ \cong}\\
        \hh_T(G/B)\ar@{^(->}[rr]^{\bigoplus\limits_{w\in W} i_w^*} &  & Q\otimes_{S}\hh_T(W),}
\]
where the top horizontal map is the embedding of the $S$-module into the $Q$-module $Q_W^*$. 
\end{enumerate}
\end{prop}
By specializing the formal group law to $F_a$ or $F_m$, respectively, and restricting $S$ to $R[\Lambda]/J_F$, we obtain a map
$\Phi^H:  \bfD_F^*\rightarrow H_T^*(G/B)$ or $\Phi^K:  \bfD_F^*\rightarrow K_T(G/B)$
 to the equivariant cohomology or equivariant K-theory. The map remains an isomorphism over the corresponding module. From now on we will not distinguish between $\bfD_{F}^*$ and $\hh_T(G/B)$. 
\begin{ex}\label{ex:Schubert} Let $X(w)=\overline{BwB/B}$ be the Schubert variety and $Y(w)=\overline{B^-wB/B}$ be the opposite Schubert variety.
For $H_T^*(G/B)$ (with $F=F_a$) or  $K_T(G/B)$ (with $F=F_m$), we write $w$ for $I_w$ since $X_{I_w}$ and $Y_{I_w}$ are independent of the reduced sequence.
\begin{enumerate}
\item \cite[\textsection 1.2]{GK19} For $H_T^*(G/B)$, $\zeta^Y_w=[X(w)]$, and $\zeta^X_w=(-1)^{\ell(w)}[X(w)]$, where each homology class is identified with its  dual cohomology class.
Then $Y_{w}^* = [Y(w)]$ and similarly $X_{w}^* = (-1)^{\ell(w)}[Y(w)]$.
\item \cite[\textsection 3]{AMSS19} For $K_T(G/B)$, $\zeta^Y_w=[\calO_{X(w)}]$ is the class of the structure sheaf of $X(w)$, $Y_w^*=[\calO_{Y(w)}(-\partial Y(w))]$, $\zeta_w^X=(-1)^{\ell(w)}[\calO_{X(w)}(-\partial X(w))]$, and $X^*_w=(-1)^{\ell(w)}[\calO_{Y(w)}]$. 
\end{enumerate}
\end{ex}

\section{Generalized Demazure operators and the generalized  Leibniz rule}\label{se:generalizedDemazure}

In this section, we generalize the operators $X_{I_v}$ and $Y_{I_v}$ on $\hh_T(G/B)$ to a more general class of elements of $Q_W$, and prove the generalized  Leibniz rule for $\bfD_F$ acting on $Q$. We use this result to compute the coproduct structure in $Q_W$, and then the product structure in $Q_W^*$.

Let $\{a_\al, b_\al \in Q:\ \al\in \Sigma\}$ be a set of elements with the property that, for all $w\in W$, 
$$
w(a_\al) = a_{w(\al)}, \quad\quad w(b_\al) = b_{w(\al)}, \text{ and }b_\al\text{ are all invertible in }Q. 
$$
For any simple root $\al$, define operators $Z_\al\in Q_W$ by
$$
Z_\al=a_\al+b_\al\de_\al.
$$
Clearly $X_\al$ and $Y_\al$ result from $Z_\al$ as special cases of $a_\al$ and $b_\al$. For any sequence $I=(i_1,...,i_k)$,
define $Z_I\in Q_W$ by
$$
Z_I=Z_{\al_{i_1}}Z_{\al_{i_2}}\cdots Z_{\al_{i_k}}.
$$
We call $Z_I$ {\it generalized Demazure operators.}

As before, we choose a reduced word expression $I_v$ for each $v\in W$.

\begin{lem} The set of generalized Demazure operators $\{Z_{I_v}\}$ forms a basis of $Q_W$ as a module over $Q$.
\end{lem}
\begin{proof}
This  follows from the fact that $b_\al\in Q$ is invertible for all simple roots $\al$ (hence, for all roots $\al$).
\end{proof}
\begin{rem}
Note that $Z_\al\in \bfD_F$ if and only it satisfies the residue condition \cite[Definition 3.7]{ZZ17}. If this is satisfied, then $Z_{I_v}\in \bfD$ and equivalently, $Z_{I_v}^*\in \bfD_F^*$. Moreover,   $Z_{I_v}$ forms a basis of $\bfD_F$ if and only if $\frac{1}{b_\al}\in S$ for all $\al$. For example, this holds for $X_\al, Y_\al$, but fails for $T_\al$ considered in Section 6 and 7. This is precisely why the stable basis is only a basis after localization.
\end{rem}

\begin{lem}\label{lem:Zcoeff}
For any sequence $J$, define coefficients $c_{J, I_w}\in Q$ by
\begin{equation}\label{eq:Zcoeff}
Z_J=\sum_{w\in W}c_{J, I_w}Z_{I_w}, 
\end{equation}
Then $c_{J,I_w}=0$ unless $w\le \wt \prod J$. 
\end{lem}

\begin{proof} 
Clearly $Z_\al = a_\al+b_\al\delta_\al$ has support on $\{w:\ w\leq s_\al\}$.  An immediate observation of the product in $Q_W$ shows inductively that $Z_J$ may be expressed as a $Q$-linear combination of $\de_v$ for $v\leq \wt\prod J$.  

For any $v\in W$ and reduced sequence $I_v=(i_1,\dots, i_k)$, let $\ga_j = \al_{i_j}$ for $j=1,\dots, k$. The coefficient of $\de_v$ in $Z_{I_v}$ is 
$$
b_{\ga_1} s_{\ga_1}(b_{\ga_2}) s_{\ga_1}s_{\ga_2}(b_{\ga_3})\dots s_{\ga_1}\dots s_{\ga_{k-1}} (b_{\ga_k}).
$$
 In particular, since $b_{\ga_j}$ is invertible, so is $w(b_{\ga_j})$ for any Weyl group element $w$, and thus the coefficient of $\de_v$ in $Z_{I_v}$ is nonzero.

Let $A = \{w\in W:  c_{J,I_w}\neq 0 \mbox{ and } w\not \leq \wt \prod J\}$, and assume $A$ is nonempty. Pick $v\in A$ to be a maximal element of $A$ in the Bruhat order. By support considerations, the only terms contributing to the coefficient of $\de_v$ in \eqref{eq:Zcoeff} is $c_{J,I_v}Z_{I_v}$. Since the coefficient of $\de_v$ in $Z_{I_v}$ is a unit, we conclude $c_{J,I_v}=0$, contrary to assumption.
\end{proof}

The structure constants $c_{J,I_w}$ reflect geometric properties in some special cases (see Section~\ref{sec:cohK}). When $Z_\al =X_\al$ for all $\al$ or $Z_\al=Y_\al$ for all $\al$, and $F=F_a$, the coefficients in 
the sum \eqref{eq:Zcoeff} vanish unless $J$ is a reduced word for $w$, in which case $c_{J,I_w}=1$;  this reflects the property that the pushforward map in homology sends the orientation class $[BS_J]$ to the Schubert variety $X(w)$ when $J$ is a reduced word for $w$. 
When  $Z_\al =X_\al$ for all $\al$ or $Z_\al=Y_\al$ for all $\al$, and $F=F_m$, coefficients vanish except when the Demazure product of $J$ is $w$, which occurs exactly once and results in $c_{J,I_w}=1$.
 In this case, the K-theoretic pushforward of $[\mathcal O_{BS_J}]$ is the structure sheaf of $X(w)$ when  $w=\wt \prod J$.
More generally, $Z_J$ is an (equivariant) operator whose dual has support only on those fixed points in the Schubert variety $X(w)$, where $w=\wt \prod J$.

We have the following lemma describing the action of $Z_\al$ on a product.
\begin{lem}\label{le:Zi1} For a simple root $\al$, and $p,q\in Q$, we have 
\begin{equation*}
Z_\al\cdot (pq)=\frac{a_\al(a_\al+b_\al)}{b_\al}pq-\frac{a_\al}{b_\al}\left[\left(Z_\al\cdot p\right)q+p(Z_\al\cdot q)\right]+\frac{1}{b_\al}(Z_\al\cdot p)(Z_\al\cdot q). 
\end{equation*}
\end{lem}
\begin{proof}One just has to plug in $Z_\al=a_\al+b_\al\de_{s_\al}$ and use the definition of the action $\de_{s_\al}\cdot p=s_\al(p)$. A comparison of both sides yields the identity.
\end{proof}

The coefficients occurring in Lemma~\ref{le:Zi1} may be generalized to the case of the action of $Z_I$ on a product $pq$. 
\begin{defn}\label{def:LeibnizCoeff} For each simple root $\al$, let $Z_\al = a_\al + b_\al \de_\al$ with $a_\al, b_\al\in Q$ and $b_\al$  invertible.  
Let $I = (i_1,\dots, i_k)$ be a sequence of indices of simple roots, with $\ga_j := \al_{i_j}$ corresponding to the $j$th entry of $I$. 
For $E, F\subset \{1,\dots k\}$,
define the {\bf Leibniz coefficients} $\bfz_{E,F}^I\in Q$ by
\begin{equation}\label{eq:Leibniz}
\bfz^I_{E,F}=(B^Z_1 B^Z_2\cdots  B^Z_k)\cdot 1,
\end{equation}
 where the operators $B^Z_j\in Q_W$ are given by 
\begin{equation}\label{eq:BZ}
B_j^Z=
\begin{cases}
\frac{1}{b_{\ga_j}}\de_{\ga_j}, &\text{ if }j\in E\cap F,\\
\vspace{-.1in}\\
-\frac{a_{\ga_j}}{b_{\ga_j}}\de_{\ga_j}, &\text{ if }j\in  E \text{ or }F, \text{ but not both},\\
\vspace{-.1in}\\
a_{\ga_j}+\frac{a_{\ga_j}^2}{b_{\ga_j}}\de_{\ga_j}, &\text{ if }j\not\in E\cup F.\\
\end{cases}
\end{equation}
\end{defn}

\begin{ex}\label{ex:LeibnizCoeffXY} Let $\ga_j = \al_{i_j}$ indicate the $j$th root listed in the sequence $I$.
If $Z=X$, then using the specific choice of coefficients for the Demazure operator yields
$$
B_j^X=
\begin{cases}
-x_{\ga_j}\de_{\ga_j}, &\text{ if }j\in E\cap F,\\
\vspace{-.1in}\\
\de_{\ga_j}, &\text{ if }j\in  E \text{ or }F, \text{ but not both},\\
\vspace{-.1in}\\
X_{\ga_j}, &\text{ if }j\not\in E\cup F.\\
\end{cases}
$$
Similarly, if $Z=Y$ indicate the push-pull operators,
$$
B_j^Y=
\begin{cases}
x_{\ga_j}\de_{\ga_j}, &\text{ if }j\in E\cap F,\\
\vspace{-.1in}\\
\frac{x_{\ga_j}}{x_{-\ga_j}}\de_{\ga_j}, &\text{ if }j\in  E \text{ or }F, \text{ but not both},\\
\vspace{-.1in}\\
\frac{1}{x_{-\ga_j}} +\frac{x_{\ga_j}}{(x_{-\ga_j})^2}\de_{\ga_j} , &\text{ if }j\not\in E\cup F.\\
\end{cases}
$$
\end{ex}

Now we prove the main technical result of this paper, generalizing \cite[Lemma 4.8]{CZZ1}. 

\begin{thm}[Generalized Leibniz Rule]\label{thm:Leibniz} 
Let $Z_I$ be a generalized Demazure operator for $I=(i_1,...,i_k)$, and let $\ga_j = \al_{i_j}$ denote the $j$th simple root in the list. Then for any $p,q\in Q$, 
\[
Z_I\cdot (pq)=\sum_{E,F\subset [k]}\bfz^I_{E,F}(Z_{E}\cdot p)( Z_{F}\cdot q), 
\]
where $\bfz^I_{E,F}$ are the Leibniz coefficients defined in \eqref{eq:Leibniz}
\end{thm}
\begin{proof}
For any simple root $\al$, observe the following two identities:
\begin{align}
&a_{\al}(1-\de_{\al})+\frac{a_{\al}(a_{\al}+b_{\al})}{b_{\al}}\de_{\al}=a_{\al}+\frac{a_{\al}^2}{b_{\al}}\de_{\al}=\frac{a_{\al}}{b_{\al}}Z_\al\de_\al ,\label{eq:case3}
\\
&Z_\al\cdot (pq)=a_\al(p-s_\al(p))q+s_\al(p)(Z_\al\cdot q). \label{eq:Zi2}
\end{align}
We prove the theorem by induction on $k$. If $k=1$, the theorem holds by Lemma~\ref{le:Zi1}.

Now assume it holds for all $I$ with $\ell(I)<k$, and let $I=(i_1,...,i_{k})$. Let $J=(i_2,...,i_{k})$ and let $\al=\al_{i_1}$.
We have
\begin{align*}
Z_I\cdot(pq)&=(Z_\al Z_J)\cdot (pq) = Z_\al\cdot ( Z_J\cdot (pq))\\
&=Z_\al\cdot \left[\sum_{E,F\subset \{2,\dots, k\}}\bfz^{J}_{E,F}(Z_{E}\cdot p) (Z_{F}\cdot q)\right]\\
&=\sum_{E,F\subset  \{2,\dots, k\}}a_\al\left[\bfz^{J}_{E,F}-s_\alpha(\bfz^{J}_{E,F})\right](Z_{E}\cdot p)(Z_{F}\cdot q)\\
&\phantom{"="}+\sum_{E,F\subset  \{2,\dots, k\}}s_\al(\bfz^{J}_{E,F})Z_\alpha\cdot \left[(Z_{E}\cdot p)(Z_{F}\cdot q)\right] \mbox{ by Equation \eqref{eq:Zi2}}\\
&=\sum_{E,F\subset  \{2,\dots, k\}}a_\al\left[\bfz^{J}_{E,F}-s_\al(\bfz^{J}_{E,F})\right](Z_{E}\cdot p)(Z_{F}\cdot q)\\
&+\sum_{E,F\subset  \{2,\dots, k\}}s_\al(\bfz^{J}_{E,F})\frac{a_\al(a_\al+b_\al)}{b_\al}(Z_{E}\cdot p)(Z_{F}\cdot q)\\
&-\sum_{E,F\subset  \{2,\dots, k\}}s_\al(\bfz^{J}_{E,F})\frac{a_\al}{b_\al}\left[(Z_\al Z_{E}\cdot p)](Z_{F}\cdot q)+(Z_{E}\cdot p)(Z_\al Z_{F}\cdot q)\right]\\
&+\sum_{E,F\subset  \{2,\dots, k\}}s_\al(\bfz^{J}_{E,F})\frac{1}{b_\al}(Z_\al Z_{E}\cdot p)(Z_\al Z_{F}\cdot q) \mbox{ by Lemma \ref{le:Zi1}} \\
&=\sum_{E,F\subset  \{2,\dots, k\}}\left[
\left(a_\al+ \frac{a_\al^2}{b_\al}\de_\al\right)\cdot
\bfz^{J}_{E,F}
\right](Z_{E}\cdot p)(Z_{F}\cdot q)\\
&
-\sum_{E,F\subset  \{2,\dots, k\}}\left(\frac{a_\al}{b_\al}\de_\al\cdot \bfz^{J}_{E,F}\right)\left[(Z_\al Z_{E}\cdot p)(Z_{F}\cdot q)+(Z_{E}\cdot p)(Z_\al Z_{F}\cdot q)\right]\\
&+\sum_{\tiny E,F\subset  \{2,\dots, k\}}\left(\frac{1}{b_\al}\de_\al\cdot \bfz^{J}_{E,F}\right)(Z_\al Z_{E}\cdot p)(Z_\al Z_{F}\cdot q) \mbox{ by Equation \eqref{eq:case3}.}
\end{align*}
Comparing the coefficients with $B^Z_{1}\cdot (\bfz^{J}_{E,F})$ from  \eqref{eq:BZ}, we see that they coincide. The proof then follows by induction.
\end{proof}

The following corollary follows immediately. We see in Section~\ref{se:restriction} that the Leibniz coefficients $\bfz^{I}_{[k],E}$ arise as factors in summands of specific structure constants in Schubert calculus, justifying the name. Here $[k]=\{1,2,...,k\}$.
\begin{cor}\label{cor:cons}[Generalized Billey's Formula]
Let $I=(i_1,...,i_k)$ be a sequence of indices of simple roots, and denote 
$m_j = s_{i_1}s_{i_2}\cdots s_{i_{j-1}}(a_{i_j})$ and $n_j = s_{i_1}s_{i_2}\cdots s_{i_{j-1}}(b_{i_j})$. For  $E\subset[k]$, we have
$$
\bfz^{I}_{[k],E}=\bfz^{I}_{E, [k]}=(-1)^{k-\vert E\vert} \prod\limits_{j\in [k]\setminus E} m_j \prod\limits_{j\in [k]} n_j^{-1}.
$$
\end{cor}

 As a consequence of Theorem~\ref{thm:Leibniz},  \cite[Proposition 9.5]{CZZ1} and the coproduct defined in Equation~\eqref{eq:coprod}, we obtain the following theorem.
\begin{thm}\label{thm:coprod}
Let $Z_\al=a_\al+b_\al\de_\al\in Q_W$ with $b_\al$ invertible, then for any $I=(i_1,...,i_k)$, we have
$$
\tr(Z_I)=\sum_{E,F\subset[k]}\bfz^I_{E, F}Z_{E}\otimes Z_{F},
$$
where $\bfz^I_{E,F}$ are defined in Definition \ref{def:LeibnizCoeff}. 
\end{thm}

We  specialize Theorem~\ref{thm:Leibniz} to the elements $X_I$ and $Y_I$. For any index $j$, the operators $B^X_j$ and $B^Y_j$ preserve $S$ under the action of $Q_W$ on $Q$, and thus  $B^X_j, B^Y_j \in \bfD_F$ (see \cite[Remark 7.8]{CZZ1}). The first statement in the next corollary is the result \cite[Proposition 9.5]{CZZ1}. 

\begin{cor} For the  Demazure elements $X_\al$, and $I=(i_1,...,i_k)$,  we have
\[
X_I\cdot (pq)=\sum_{E,F\subset [k]}\bfx^I_{E,F}(X_E\cdot p) (X_{F}\cdot q), 
\]
where $\bfx^I_{E,F}=(B^X_1 B^X_2\cdots  B^X_k)\cdot 1$ with $B^X_j\in \bfD_F$ defined  in Example~\ref{ex:LeibnizCoeffXY}. Similarly, for the  push-pull elements 
$Y_\al$ and $I=(i_1,...,i_k)$, we have
\[
Y_I\cdot (pq)=\sum_{E,F\subset [k]}\bfy^I_{E,F}(Y_{E}\cdot p) (Y_{F}\cdot q), 
\]
where $\bfy^I_{E, F}=(B^Y_1 B^Y_2\cdots  B^Y_k)\cdot 1$, and $B^Y_j\in \bfD_F$ is defined 
in Example~\ref{ex:LeibnizCoeffXY}.
\end{cor}

\junk{
\begin{cor}
Let $I=(i_1,...,i_k)$ be a sequence of indices of simple roots, and denote 
$\be_j=s_{i_1}s_{i_2}\cdots s_{i_{j-1}}(\al_{i_j})$. 
 For  $E\subset[k]$, we have
 \begin{align*}
\bfx^{I_w}_{[k], E}&=\bfx^{I_w}_{E, [k]}=(-1)^{|E|}\prod\limits_{j\in E}x_{\be_{j}}, \mbox{and}\\
\bfy^{I_w}_{[k], E}&=\bfy^{I_w}_{E, [k]}=(-1)^{k-|E|}\prod\limits_{j\in [k]} x_{\be_{j}}(\prod\limits_{j\in [k]\setminus E} x_{-\be_{j}})^{-1}.
\end{align*}
\end{cor}
}

\section{The structure constants of equivariant oriented cohomology of flag varieties}\label{sec:prod}

In this section we prove the main result, i.e., the formulas of  structure constants of $Z_{I_w}^*$ in $\hh_T(G/B)$, with resulting formulas for the structure constants of $X_{I_w}^*$ and  of $Y_{I_w}^*$.

Let $\{Z_{I_w}^*\}$ be the basis of $Q_W^*$ (as a module over $Q$) dual to the basis $\{Z_{I_w}\}$ of $Q_W$ introduced in Section~\ref{se:generalizedDemazure}.
\begin{thm}\label{thm:prod}
For any $u,v\in W$, the product $ Z_{I_u}^*\ Z_{I_v}^*$ is given by
\[
Z_{I_u}^*\ Z_{I_v}^*=\sum_{w\ge u, v}\bbz^{I_w}_{I_u,I_v}Z^*_{I_w},
\]
where 
$$
\bbz^{I_w}_{I_u, I_v}=\sum_{E, F\subset [\ell(w)]}\bfz^{I_w}_{E,F} c_{E,I_u} c_{F,I_v}\in Q,
$$
 $\bfz^{I_w}_{E, F}\in Q$ are the Leibniz coefficients given in Definition~\ref{def:LeibnizCoeff}. As before, the $Q$ elements $c_{E,I_u}$ and $c_{F,I_v}$ are defined as constants appearing in the expansion
 \begin{equation}
Z_J=\sum_{w\in W}c_{J, I_w}Z_{I_w}. 
\end{equation}
\end{thm}

\begin{ex}Consider the $A_3$-case. Consider $I_u=(2,3,1,2,1), I_v=(1,2,3,2,1)$, then $\bbz^{I_w}_{I_u, I_v}=0$ unless $w=w_0$  is the longest element. Fix $I_{w_0}=(1,2,3,1,2,1)$, in which case we have 
\begin{align*}
&\bfz^{I_{w_0}}_{\{2,3,4,5,6\}, \{1,2,3,5,6\}}&&=B^Z_1B_2^ZB_3^ZB_4^ZB_5^ZB^Z_6\cdot 1 \\
&&&=\frac{a_{\al_1}a_{\al_2}}{b_{\al_1}b_{\al_2}b_{\al_1+\al_2}b_{\al_2+\al_3}b_{\al_1+\al_2+\al_3}},
\end{align*}
and $c_{\{2,3,4,5,6\}, I_u}=c_{\{1,2, 3, 5,6\}, I_v}=1$.
Therefore,
\[Z^*_{I_u}\cdot Z^*_{I_v}=\frac{a_{\al_1}a_{\al_2}}{b_{\al_1}b_{\al_2}b_{\al_1+\al_2}b_{\al_2+\al_3}b_{\al_1+\al_2+\al_3}} Z^*_{I_{w_0}}. \]
\end{ex}

\begin{proof}[Proof of Theorem~\ref{thm:prod}]
The coproduct structure $\Delta$ on $Q_W$ (Equation~\eqref{eq:coprod}) naturally induces a product on $Q_W^*$. 
For all $f, g \in Q_W^*$ and $\sum_{w\in W} q_w \de_w \in Q_W$, 
\begin{align*}
\langle f\ g, \sum\limits_{w\in W} q_w\de_w \rangle  & = \langle f \otimes g, \Delta(\sum_{w\in W} q_w\de_w ) \rangle  \\
&=  \langle f \otimes g,\sum_{w\in W} q_w\de_w \otimes \de_w \rangle\\
& = \sum_{w\in W} q_w   \langle f, \de_w  \rangle   \langle g, \de_w  \rangle.
\end{align*}
Note that this product corresponds to the product on $Q_W^*$ introduced at the beginning of Section~\ref{se:dualopalg} since
\begin{align*}
\langle f_u\ f_v, \sum\limits_{w\in W}  q_w\de_w\rangle &=
\begin{cases}
\langle f_u, \sum q_w\de_w\rangle =  \sum_w q_w\langle f_u, \de_w  \rangle , &\mbox{if $u=v$};\\
0, & \mbox{otherwise},
\end{cases} \\
&= \begin{cases}
q_u, & \mbox{if $u=v$};\\ 
0 ,& \mbox{otherwise}.
\end{cases}
\end{align*}

From Theorem \ref{thm:coprod} we have
\begin{align*}
\tr (Z_{I_w})&=\sum_{E,F\subset [\ell(w)]}\bfz^{I_w}_{E,F}Z_{E}\otimes Z_{F}\\
&=\sum_{E,F\subset [\ell(w)]}\bfz^{I_w}_{E,F}\left[\left(\sum_{u\in W}c_{E, I_u}Z_{I_u}\right)\otimes\left(\sum_{v\in W}c_{F, I_v}Z_{I_v}\right)\right]\\
&=\sum_{u,v\in W}\left[\sum_{E,F\subset [\ell(w)]} \bfz^{I_w}_{E,F} c_{E, I_u} c_{F, I_v}\right] Z_{I_u}\otimes Z_{I_v}\\
&=\sum_{u,v\in W}\bbz^{I_w}_{I_u, I_v}Z_{I_u}\otimes Z_{I_v}.
\end{align*}

Finally we obtain the coefficient by calculating the pairing:
\begin{align*}
\langle Z_{I_u}^*\ Z_{I_v}^*, Z_{I_w}\rangle & =  \langle Z_{I_u}^*\otimes Z_{I_v}^*, \Delta(Z_{I_w})\rangle = \bbz^{I_w}_{I_u, I_v}.
\end{align*}
Let $I_w\vert_E$ be the subsequence obtained from restricting $I_w$ to $E$. 
Since $w=\wt\prod I_w\ge \wt\prod ( I_w\vert_E)$ for any $E\subset [\ell(w)]$,  
by Lemma~\ref{lem:Zcoeff},  $\bbz^{I_w}_{I_u, I_v}=0$ unless $u\le w$ and  $v\le w.$ 
\end{proof}

The coproduct structure on the left $Q$-module $Q_W$ restricts to a coproduct structure on the left $S$-module $\bfD_F$ \cite[Theorem 9.2]{CZZ1}.
 Consequently, the embedding $\bfD_F^*\subset Q_W^*$ is an embedding of subrings. So the structure constants of the $S$-bases $\{X_{I_w}^*\}$ and $\{Y_{I_w}^*\}$ in $\bfD_F^*$ are precisely those of the $Q$-bases $\{X_{I_w}^*\}$ and $\{Y_{I_w}^*\}$ in $Q_W^*$.

Specializing Theorem \ref{thm:prod} to the $X$-operators, we have
\[
X_{I_u}^*\ X_{I_v}^*=\sum_{w\ge v, w\ge u}\bbx^{I_w}_{I_u, I_v}X^*_{I_w}, 
\]
with 
\begin{align}\label{eq:generalbbx}
\bbx^{I_w}_{I_u, I_v}=\sum_{E, F\subset [\ell(w)]}\bfx^{I_w}_{E,F} c_{I_w \vert_{E}, I_u} c_{I_w \vert_{F}, I_v},
\end{align}
where $c_{I,I_v}$ are the coefficients that occur in the expansion $X_I = \sum_v c_{I,I_v} X_{I_v}$. 
 It follows from \cite[Theorem 9.2 and Proposition 7.7]{CZZ1} that $\bfx^{I_w}_{E,F}\in S$, that $c_{I,I_w}\in S$, so $\bbx^{I_w}_{I_u,I_v} \in S$. 
 Similarly, specializing to the $Y$-operators, the structure constants for $Y_{I_w}^*$ are denoted by $\bby^{I_w}_{I_v,I_u}$ and can be expressed as
$$
\bby^{I_w}_{I_u, I_v}=\sum_{E, F\subset [\ell(w)]}\bfy^{I_w}_{E,F} c_{I_w \vert_{E}, I_u} c_{I_w \vert_{F}, I_v},
$$
where now the coefficients $c_{I, I_v}$ are those appearing in the expansion of $Y_I$.  
 As before, $\bfy^{I_w}_{E, F}\in S$ and $c_{I,I_w}\in S$, so $\bby^{I_w}_{I_u,I_v}\in S$. 
In \textsection\ref{sec:cohK} we show that these coefficients simplify in the case that $F=F_a$ or $F=F_m$, resulting in Theorem 1 from \cite{GK19}. It is worth noting that the formula \eqref{eq:generalbbx} can be used to prove the  Leray-Hirsch Theorem for flag varieties (see \cite{DZ20}).

\begin{ex}Assume the root datum is of type $A_1$, then $W=\{e, s_1\}$. 
We calculate the basis change explicitly:
$$
X_e^*=f_e+f_{s_1}, \quad X_{(1)}^*=-x_{\al_1}f_{s_1}, \\
$$
and then we may obtain the products directly:
\begin{align*}
&X_e^*\ X_e^*=X_e^*, && X_e^*\ X_{(1)}^*=X_{(1)}^*,  && X_{(1)}^*\ X_{(1)}^*=-x_{\al_1}X_{(1)}^*, 
\end{align*}
and note that it agrees with Theorem  \ref{thm:prod} with $Z=X$.
\end{ex}

\begin{ex}
Consider the root datum $A_2$, with $W=\{e, s_1, s_2, s_1s_2, s_2s_1, w_0\}$. For the longest element $w_0$, we fix the reduced sequence $I_{w_0}=s_1s_2s_1$. 

We use the calculation in Example~\ref{ex:dualbasischange}, and the product structure on $Q_W^*$ to obtain the multiplication table for $\{X_{I_v}\}$. Recall that $f_uf_v = 1$ if $u=v$ and $0$ otherwise, and that $X_e^* = f_e+f_{s_1} + f_{s_2} +f_{s_1s_2} + f_{s_2s_1} + f_{w_0}$. If $X_w = \sum_u a_u f_u$, we have
$$X_w^*X_e^* = \left(\sum_u a_u f_u\right)\left(\sum_v f_v\right) =\sum_u a_u f_u=X_w^*$$ 
for all $w\in W$. 
Similarly,
\begin{align*}
X_{I_{w_0}}^*\ X^*_{(1,2)} &= (-x_{\al_1}x_{\al_2}x_{\al_{13}}f_{w_0})\left(x_{\al_1}x_{\al_{13}}(f_{s_1s_2}+f_{w_0})\right)\\
&=-x_{\al_1}^2 x_{\al_2}x_{\al_{13}}^2 f_{w_0}\\
&= x_{\al_1}x_{\al_{13}}X_{I_{w_0}}^*.
\end{align*}
The other products are as follows:
\begin{table}[htp]
\begin{center}
\begin{tabular}{l l l c l l l}
$X_{I_{w_0}}^*\ X^*_{I_{w_0}}$ & = & $-x_1x_2x_{\al_{13}}X_{I_{w_0}}^*$  
& &
$X_{I_{w_0}}^*\  X^*_{(2,1)}$ & $=$ & $x_2x_{\al_{13}}X_{I_{w_0}}^*$
\\
$X_{(1)}^*\  X_{(2)}^*$ & $=$ & $X_{(1,2)}^*+X_{(2,1)}^*+\kappa_1X_{I_{w_0}}^*$
& &
$X_{I_{w_0}}^*\  X_{(2)}^*$ & $=$ & $-x_{\al_{13}}X_{I_{w_0}}^*$
\\
$X_{(2)}^*\  X_{(2)}^*$&$=$& $\frac{x_{\al_{13}}-x_2}{x_1}X^*_{(1,2)}-x_2X_{(2)}^*$,
& & 
$X_{(2,1)}^*\  X_{(2,1)}^*$ & $=$ & $x_2x_{\al_{13}}X_{(2,1)}^*$
\\
$X_{(2,1)}^*\  X_{(2)}^*$ & $=$ &$\frac{x_{\al_{13}}-x_2}{x_1}X_{I_{w_0}}^*-x_2X_{(2,1)}^*$, 
& &
$X_{I_{w_0}}^*\  X^*_{(1)}$ & $=$ & $-yX_{I_{w_0}}^* $ 
\\
$X_{(1)}^*\  X_{(1)}^*$ & $=$ & $-x_1X_{(1)}^*+\frac{x_{\al_{13}}-x_1}{x_2}X_{(2,1)}^*$
 & & 
$X_{(1,2)}^*\  X_{(1,2)}^*$ & $=$ & $x_1x_{\al_{13}}X_{(1,2)}^*$
\\
& &$+ \frac{x_1y+x_{\al_{13}}^2-x_1x_{\al_{13}}-y^2}{x_1x_2x_{\al_{13}}}X_{I_{w_0}}^*.$
\end{tabular}
\end{center}
\end{table}%
Here $y$ was defined in Example~\ref{ex:dualbasischange}. 

One can check that the above coefficients $\bbx_{I_u, I_v}^{I_w}$ agree with the formula \eqref{eq:generalbbx}. 
Note that when computing $\bbx^{I_{w_0}}_{1,1}$, one needs to compute the following coefficients:
\[
\bfx^{I_{w_0}}_{\{3\}, \{3\}}, ~\bfx^{I_{w_0}}_{\{1, 3\}, \{3\}},~\bfx^{I_{w_0}}_{\{3\}, \{1, 3\}},~ \bfx^{I_{w_0}}_{\{1, 3\},\{1,3 \}}.
\]
\end{ex}

As an application, we consider the case of a partial flag variety. Let $K$ be a subset of $[n]$. Let $P_K$ be the standard parabolic subgroup, $W_K<W$ the corresponding subgroup, and $W^K\subset W$ be  the set of minimal length representatives of $W/W_K$.  We say a set of reduced sequences $I_w$ is $K$-compatible if for each $w=uv, u\in W^K, v\in W_K$, we have $I_w=I_u\cup I_v$, i.e., $I_w$ is the concatenation of $I_u$  with $I_v$.

\begin{thm} \label{thm:K} Suppose the set $\{I_w\}$ is $K$-compatible. Then  for any $v,u\in W^K$, we have
\[
X_{I_u}^*\ X_{I_v}^*=\sum_{w\in W^K, w\ge v, w\ge u} \bbx^{I_w}_{I_u, I_v}X^*_{I_w}. 
\]
\end{thm}
\begin{proof} It follows from \cite[Corollary 8.4]{CZZ2} that $X^*_{I_u}, u\in W^K$ is a basis of $(Q_W^*)^{W_K}$. Moreover, from Lemma 4.3 of loc.it., we know $\de_w\bullet (ff')=(\de_w\bullet f)(\de_w \bullet f')$. Therefore, $X_{I_u}^*\ X_{I_v}^*\in (Q_W^*)^{W_K}$, so is a linear combination of $X_{I_w}^*, w\in W^K$. 
\end{proof}

Geometrically, under the assumption of this theorem, it follows from \cite[Corollary 8.4]{CZZ2} that  $\{X_{I_w}^*\}_{w\in W^K}$ is a basis of $(\bfD_F^*)^{W_K}\cong \hh_T(G/P_K)$.  So the product  $X_{I_u}^*\ X_{I_v}^*, u,v\in W^K$ is a linear combination of $X_{I_w}^*, w\in W^K$. 

\begin{cor}Let $F=F_a$ or $F_m$, and suppose $u\in W$ satisfies that $u\in W^K$ for some $K$ and $u$ is the longest element in $W^K$. Then for any $v\in W^K$, $\bbx^{w}_{u,v}=0$ for any $w\in W$, unless $w=u$. 
\end{cor}
\begin{proof}In these cases, the braid relations are satisfied, so the structure constants do not depend on the choice of reduced sequences. In other words, fixing $u$ and $K$, we can assume we have chosen  $K$-compatible reduced sequences. Then Theorem \ref{thm:K} applies, which implies that for any $v\in W^K, w\in W$, we have $\bbx^{I_w}_{u,v}=0$ unless $w\in W^K$ and $w\ge u$. Since $u$ is maximal in $W^K$, so $w=u$. 
\end{proof}

\section{Structure constants in singular cohomology and K-theory}\label{sec:cohK}

We restrict our attention to $H_T^*(G/B)$ and $K_T(G/B)$ to recover formulas in \cite{GK19} of structure constants of Schubert classes   for singular cohomology ($F=F_a$) and K-theory ($F=F_m$). We first simplify the coefficients $c^X_{I,I_w}$ and $c^Y_{I,I_w}$ in these two cases. 
Recall that, when the formal group law is $F=F_a$ or $F=F_m$, the braid relations are satisfied for $Z_\al=X_\al$ and $Z_\al=Y_\al$. We consider the equivariant oriented cohomology together with either the additive or multiplicative formal group law, and restrict the coefficient ring to $S^a$ or $S^m$. 

\begin{lem}\label{lem:c}
Let $J$ be a word in the Weyl group. As in Lemma~\ref{lem:Zcoeff}, define coefficients $c_{J,I_w}$  by 
$$
Z_J= \sum_{w\in W}c_{J,I_w}Z_{I_w}.
$$
\begin{enumerate}
\item Let  $F=F_a$. If $Z_\al=X_\al$ or $Z_\al=Y_\al$, then
$$
c_{J,I_w} = \begin{cases}
1 ,&\mbox{if $J$ is  a reduced word for $w$};\\
0, &\mbox{else}.
\end{cases}
$$
\item Let  $F=F_m$. If $Z_\al=X_\al$ or $Z_\al=Y_\al$, then
$$
c_{J,I_w} = \begin{cases}
1 ,&\mbox{if $w=\wt \prod J$};\\
0, &\mbox{else}.
\end{cases}
$$
\end{enumerate}
\end{lem}

\begin{proof} When $F=F_a$ or $F=F_m$, it is well-known that the braid relations are satisfied. We write  $c_{J,w}$ for the coefficient $c_{J,I_w}$.
 When $F=F_a$, $Z_\al^2=0$, so if $J$ is not reduced,  $Z_J=0$. If $J$ is reduced and $\prod J=w$,  then $Z_J=Z_{w}$, so $c_{J,w}=1$ and $c_{J,v}=0$ for $v\neq w$. 

When $F=F_m$, we have $Z_\al^2=Z_\al$ and thus $Z_J=Z_w$ where $w:={\wt \prod J}$. It follows that $c_{J,w}=1$ and $c_{J,v}=0$ for $v\neq w$. 
\end{proof}

\begin{ex}\label{ex:FaSch}
For $H^*(G/B)$ and $F=F_a$, as described in Example~\ref{ex:Schubert} and Proposition 2.7, the element $\zeta_w^X$ in $\bfD_F^*$ corresponds under a natural isomorphism
$$
\bfD_{F_a}^*\longrightarrow \hh_T(G/B)
$$
to the equivariant cohomology class Poincar\'e dual to $[X(w)]$, where $[X(w)]$ is the homology class of the Schubert variety. Furthermore, the first Chern classes of the corresponding line bundles are $x_{\al} = \al$ for all simple roots $\al$.

For each $w\in W$, fix a reduced sequence $I_w$. From the specialization of Theorem \ref{thm:prod}, we have defining relations
$$
Y_{u}^*\ Y_v^*=\sum_{w\ge u, w\ge v}\bby^{I_w}_{v,u}Y_{w}^*
$$
for  $\bby^{I_w}_{v,u}$ . Then
\begin{align*}
\bby^{I_w}_{u, v}&=\sum_{E,F\subset [\ell(w)]}\bfy^{I_w}_{E, F}c^Y_{E, I_u}c^Y_{F, I_v}\text{ by Theorem~\ref{thm:prod},}\\
&=\sum_{E,F\tiny\text{ reduced}\atop \mbox{for }u,v} \bfy^{I_w}_{E, F}, \text{ by Lemma~\ref{lem:c}}(1)
\end{align*}
where the second sum is over $E,F$ whose corresponding products of reflections are reduced and equal to $u, v$ respectively. 
 Recall that
\[
\bfy^{I_w}_{E, F}=(B^Y_{{1}}B^Y_{{2}}\cdots B^Y_{{\ell(w)}})\cdot 1,
\]
with 
\[
B_j^Y=\left\{\begin{array}{ll}
x_{\be_j}\de_{\be_j}, &\text{ if }j\in E\cap F,\\
\de_{\be_j}, &\text{ if }j\in  E \text{ or }F, \text{ but not both},\\
Y_{\be_j}, &\text{ if }j\not\in E\cup F.\\
 \end{array}\right.
\]
with $\be_j = \al_{i_j}$.

The coefficients $\bby^{I_w}_{u,v}$ coincide with the structure constants $c^w_{uv}$ in \cite[Theorem 1]{GK19}. Note that in this case, $Y_\al=-X_\al$, so $\zeta^Y_w = (-1)^{\ell(w)} \zeta^X_w$, and thus $X_w^*=(-1)^{\ell(w)} Y_w^*$.  Therefore, 
\[
\bbx^{I_w}_{u,v}=(-1)^{\ell(w)+\ell(u)+\ell(v)}\bby^{I_w}_{u,v}. 
\]

\end{ex}
\begin{ex}\label{ex:FmSch} For $K_T(G/B)$ (and $F=F_m$), we have $x_\alpha = 1-e^{-\alpha}$. 
The action of $X_\alpha$ (resp. $Y_{-\alpha}$) on $K_T(pt)$ corresponds to the action of the ordinary (resp.  isobaric) Demazure operator in  \cite{GK19}.

Fixing a reduced sequence $I_w$ for each $w$, we have \[
X_u^*\ X_v^*=\sum_{w\ge u, w\ge v}\bbx^{I_w}_{u,v}X_w^*=\sum_{w\ge v, w\ge u}\sum_{E,F}\bfx^{I_w}_{E,F}X_w^*,
\]
where by Lemma~\ref{lem:c}(2), the second sum is over all $E,F\subset [\ell(w)]$ such that $\wt \prod E = u$ and $\wt \prod F=v$.
Here, we have
\[
B_j^X=\left\{\begin{array}{ll}
-(1-e^{-\be_j})\de_{\be_j}, &\text{ if }j\in E\cap F,\\
\de_{\be_j}, &\text{ if }j\in  E \text{ or }F, \text{ but not both},\\
X_{\be_j}, &\text{ if }j\not\in E\cup F,\\
 \end{array}\right.
\] where $\be_j = \al_{i_j}$.

The classes $\{\xi_w:\ w\in W\}$ in \cite{GK19} are defined as the dual basis to $[\mathcal{O}_{X(w)}(-\partial X(w)) ]$ under the pairing obtained by taking the equivariant cap product and pushing forward to a point. Each $\xi_w$ coincides with the Poincar\'e dual class to $[\mathcal{O}_{Y(w)}]$.  In Example~\ref{ex:Schubert} we note that $X_w^* =(-1)^{\ell(w)} [\mathcal{O}_{Y(w)}]$, and thus $\xi_w=(-1)^{\ell(w)} X_w^*$. Therefore,
\[
\xi_u\ \xi_v=(-1)^{\ell(u)+\ell(v)+\ell(w)}\sum_{w\ge u, v}\bbx^{I_w}_{u,v}\xi_w.
\]
It follows that the coefficients $(-1)^{\ell(u)+\ell(v)+\ell(w)}\bbx^{I_w}_{u,v}$ coincide with $a^w_{uv}$ in  \cite{GK19}, as is clear from the formula.

With the observation that the classes 
$\{{\mathring \xi}_w:\ w\in W\}$ defined in \cite{GK19} satisfy ${\mathring\xi}_w = Y_w^*$,  a similar argument 
implies that $\bby^w_{u,v}$ coincide with the structure constants ${\mathring a}^w_{u,v}$  defined in \cite{GK19}.

\end{ex}

\begin{ex}Let $F=F_a$. Consider the  $A_2$ case. If $I_w=s_1s_2s_1, u=s_1, v=s_1s_2$, then  
\[
\bby^{s_1s_2s_1}_{s_1,s_1s_2}=\bfy^{s_1s_2s_1}_{\{1\}, \{1,2\}}+\bfy^{s_1s_2s_1}_{\{3\}, \{1,2\}}=\begin{pmatrix}
\al_1 \de_1&\de_2& Y_1\\
+\de_1& \de_2&\de_1
\end{pmatrix}\cdot 1=0+1=1. 
\]
Similarly, 
\[\bby^{s_1s_2s_1}_{s_1,s_2s_1}=\bfy^{s_1s_2s_1}_{\{1\}, \{2,3\}}+\bfy^{s_1s_2s_1}_{\{3\}, \{2,3\}}=
\begin{pmatrix}
\de_1&\de_2&\de_1\\
+Y_1&\de_2&\al_1\de_1
\end{pmatrix}\cdot 1=1-1=0.\]

For the $A_3$ case, one can also compute 
\begin{align*}
\bby^{s_1s_2s_3s_1s_2}_{s_2s_3s_2, s_1s_2s_1}&=\bfy^{s_1s_2s_3s_1s_2}_{\{2,3,5\}, \{1,2,4\}}+\bfy^{s_1s_2s_3s_1s_2}_{\{2,3,5\}, \{2,4,5\}}\\
&=\begin{pmatrix}
\de_1&\al_2\de_2&\de_3&\de_1&\de_2\\
+Y_1&\al_2\de_2 &\de_3&\de_1 & \al_2\de_2
\end{pmatrix}\cdot 1\\
&=(\al_1+\al_2)+\al_3. 
\end{align*}
\end{ex}
\begin{ex}Let $F=F_m$. Consider the $A_3$ case, with	 $I_w=s_1s_2s_3s_1s_2, u=s_2s_3s_2, v=s_1s_2s_1$. We have
\begin{align*}
\bbx^{s_1s_2s_3s_1s_2}_{s_2s_3s_2,s_1s_2s_1}&=\bfx^{s_1s_2s_3s_1s_2}_{\{2,3,5\}, \{1,2,4\}}+\bfx^{s_1s_2s_3s_1s_2}_{\{2,3,5\}, \{2,4, 5\}}+\bfx^{s_1s_2s_3s_1s_2}_{\{2,3,5\}, \{1,2,4, 5\}}\\
&=\begin{pmatrix}
\de_1& -x_2\de_2& \de_3& \de_1& \de_2\\
+X_1& -x_2 \de_2&  \de_3 & \de_1&-x_2\de_2\\
+\de_1 &-x_2\de_2&  \de_3&  \de_1 &-x_2\de_2
\end{pmatrix}\cdot 1\\
&=-x_{1+2}+\frac{x_{1+2+3}x_2-x_{2+3}x_{1+2}}{x_1}+x_{2+3}x_{1+2}\\
&=x_{\al_2}-x_{\al_1+2\al_2+\al_3}. 
\end{align*}
\end{ex}

\section{Structure constants of cohomological stable bases}
\label{sec:cohstab}

In this section, we let $F=F_a$ and $R=R^a=\bbZ[h]$.
 We recall the definition of  the cohomological stable basis of Maulik-Okounkov, and generalize Su's formula of structure constants for Segre-Schwartz-MacPherson classes (Theorem~\ref{thm:cohstab}). We use the twisted group algebra language for singular cohomology, whose $K$-theory version was given  in \cite{ SZZ17}. As the framework and proofs are very similar to earlier sections,  we will only review essential properties.  Some of the notation  introduced below is restricted to this section only.

Let $R^a=\bbZ[h]$, $S^a=\Sym_{R^a}(\La)$ and $Q^a=\Frac(S^a)$. Define
$$
Q^a_W=Q^a\rtimes_{R^a} R^a[W]
$$ with $Q^a$-basis $\de_w, w\in W$.  For simplicity we introduce the following notation:
\[
\widehat \al=h-\al, \quad \al_{w_0}=\prod_{\al>0}\al, \quad \widehat\al_{w_0}=\prod_{\al>0}(h-\al). \]

Finally, for any simple root $\al$, define 
 an operator associated to this root by
\[
T_\al=-h\frac{1}{\al}(1-\de_{\al})-\de_{\al}=-\frac{h}{\al}+\frac{\widehat\al}{\al}\de_\al\in Q_W^a.
\]
By direct computation, the set $\{T_\al\}_{\al\in \{\al_1,\dots, \al_n\}}$ satisfies the braid relations,
and $T_\al^2=1$. Indeed, the algebra generated by $\{T_\al \}$ is called the degenerate (or graded) Hecke algebra. Note that $T_\al$ is a special case of $Z_\al$, occurring over $R=R^a$.

 For any sequence $I = (i_1,\dots, i_\ell)$ (not necessarily reduced), we define the {\it Demazure-Lusztig operator} 
$$
T_I = T_{\al_{i_1}}\dots T_{\al_{i_\ell}}
$$ in cohomology to be the product of the operators indicated in the list $I$.  
It follows from the relations that, if $I$ and $I'$ are two sequences with $w:=\prod I = \prod I'$, then $T_I = T_{I'}$, and we denote it $T_w$. The set $\{T_w \vert\ w\in W\}$ is a basis of $Q_W^a$.

Let $(Q_W^a)^*$ be the $Q^a$-dual of $Q_W^a$, and let $\{T_w^*\}\subseteq (Q_W^a)^*$
be the dual basis. Denote the basis of $(Q_W^a)^*$ dual to $\{\de_w\in Q_W^a\}$ by $\{f_w\}$, as in  \textsection\ref{sec:prelim}. The identity of the ring $(Q_W^a)^*$ is denoted by $\unit=\sum_{w\in W}f_w$.
The ring $Q_W^a$ acts on $(Q_W^a)^*$ via the $\bullet$-action, given as before by
$$ 
\langle z\bullet q^*, z'\rangle = \langle q^*, z'z\rangle\quad \mbox{for $z, z'\in Q_W^a, q^*\in (Q_W^a)^*$.}
$$

It induces a $W$-action on $(Q_W^a)^*$ via the embedding $W\subset Q_W^a$. Let $((Q_W^a)^*)^W$ denote the Weyl-invariant subgroup of  $(Q_W^a)^*$.

In this section only, denote by $\hY\in Q_W^a$ the element
\[
\sum_{w\in W}\de_w\frac{1}{\al_{w_0}\widehat\al_{w_0}}
=\sum_{w\in W}\de_w\frac{1}{\prod_{\al>0}\al(h-\al)}.
\]
The map $\hY\bullet \_: (Q_W^a)^*\to ((Q_W^a)^*)^{W}=Q^a\unit$ 
is the algebraic analogue of the composition of the map 
\[Q^a\otimes_{S^a}H^*_{T\times \bbC^*}(T^*G/B)\cong Q^a\otimes_{S^a} H^*_{T\times \bbC^*}(G/B)\to Q^a\otimes_{S^a}H^*_{T\times \bbC^*}(\pt),\]
where the last map is the equivariant pushforward of cohomology class on $G/B$ to a point on the second term. The proofs in  \cite[Lemma 7.1]{CZZ2} and \cite[Lemma 5.1]{SZZ17} easily extend to show that, for any $f, g \in (Q_W^a)^*$,
\[
\hY\bullet ((T_\al\bullet f)\cdot g)=\hY\bullet (f\cdot (T_\al\bullet g)).
\]

\begin{defn} We define two bases of $(Q_W^a)^*$ as a module over $Q^a$.  Let
 \begin{align*}
 \stab^+_w &=T_{w^{-1}}\bullet ( \al_{w_0}f_e), \mbox{and}\\
 \stab^-_w &=(-1)^{\ell(w_0)}T_{w^{-1}w_0}\bullet (\al_{w_0}f_{w_0}).
 \end{align*}
Then $\{\stab^+_w:\ w\in W\}$  and $\{\stab^-_w:\ w\in W\}$ each form a basis for $(Q_W^a)^*$ as a module over $Q^a$.  We call these bases the {\bf  cohomological stable bases.} See \cite{S15} for more details.
\end{defn}
It is immediate from the definition that  $\stab^+_w$ has support on $\{f_v:\ v\leq w\}$ and $\stab^-_w$ has support on $\{f_v: v\geq w\}$.

The following lemma is the analogue of Theorem 5.7 and Lemma 5.6 in \cite{SZZ17}. The first identity was due to Maulik-Okounkov originally.
\begin{lem}\label{lem:Kpairing} We have 
\[
\hY\bullet \left[\stab^+_v\cdot \stab^-_u\right]=(-1)^{\ell(w_0)}\de_{v,u}\unit, \quad \hY \bullet \left[\stab^+_v\cdot \widehat \al_{w_0}T_u^* \right]=\de_{v,u}\unit.
\]
\end{lem}

 Define structure constants $\bbt^{w}_{u,v}\in Q^a$ by the equation
 \[
\stab^-_u\cdot \stab^-_v=\sum_{w\in W}\bbt^{w}_{u,v}\stab^-_w.
\]
 We now present the main result about the stable basis
$\{\stab^-_w\}.$

\begin{thm}\label{thm:cohstab}
 The classes $\stab^-_w$ and the coefficients $\bbt^{w}_{u,v}$ satisfy the following properties:
\begin{enumerate}
\item We have $\stab^-_w=(-1)^{\ell(w_0)}\widehat\al_{w_0}T^*_w$. 
\item For each $w\in W$, fix a reduced sequence $I_w$.  Then
\[
\bbt^{w}_{u,v}={\sum_{\tiny\begin{array}{c}E,F\subset[\ell(w)]\\
\prod (I_w \vert_{E})=u, \prod (I_w \vert_{F})=v\end{array} }}\widehat\al_{w_0}^2 \bft^{I_w}_{E,F}, 
\]
 where $\bft_{E, F}^{I_w}=(B_1^T B_2^T\cdots B^T_k)\cdot 1$ with  
\[
B_j^T=\left\{\begin{array}{ll}
\frac{\al_{i_j}}{\widehat\al_{i_j}}\de_{i_j}, &\text{ if }j\in E\cap F,\\
\frac{h}{\widehat\al_{i_j}}\de_{i_j}, &\text{ if }j\in  E \text{ or }F, \text{ but not both},\\
-\frac{h}{\al_{i_j}}+\frac{h^2}{\al_{i_j}\widehat\al_{i_j}}\de_{i_j}, &\text{ if }j\not\in E\cup F.\\
 \end{array}\right.
\]
\end{enumerate}
\end{thm}
\begin{proof}(1). This follows from Lemma \ref{lem:Kpairing} above.

(2). For each $w\in W$, we fix a reduced decomposition. We have 
\[\stab^-_u\cdot \stab^-_v=(-1)^{\ell(w_0)}\widehat\al_{w_0}T_u^*\cdot (-1)^{\ell(w_0)}\widehat \al_{w_0}T_v^*
=\widehat\al_{w_0}^2T_u^*\cdot T_v^*.\]
Therefore, it suffices to consider the structure constants for $T_u^*$. But  the elements $T_u$ are an instantiation of $Z_{I_u}$ with the coefficient ring $R^a$, with $a_{\al_{i_j}} = -h/\al_{i_j}$ and $b_{\al_{i_j}}=\widehat\al_{i_j}/\al_{i_j}$. Thus Theorem \ref{thm:prod} indicates how to multiply the corresponding dual elements, resulting in $B_j^T$ defined as above.
\end{proof}
When $h=-1$,  the Demazure Lusztig operator $T_\alpha$ specializes to the operator considered by Su in \cite{S19}, allowing us to recover his formula for the structure constants from the SSM classes from Theorem~\ref{thm:cohstab}.
\begin{ex}Consider the $A_2$-case. If $I_w=s_1s_2s_1, u=v=s_1$, then 
\begin{align*}
\bbt^{s_1s_2s_1}_{s_1, s_1}&=\widehat\al_1\widehat\al_2\widehat\al_{13}(\bft^{121}_{\{1\}, \{3\}}+\bft^{121}_{\{3\}, \{1\}}+\bft^{121}_{\{1\}, \{1\}}+\bft^{121}_{\{3\}, \{3\}})\\
&=\widehat\al_1\widehat\al_2\widehat\al_{13}\begin{pmatrix}
\frac{h}{\widehat\al_{1}}\de_1&-\frac{h}{\al_2}+\frac{h^2}{\al_2\widehat\al_2}\de_2&\frac{h}{\widehat\al_1}\de_1\\
\frac{h}{\widehat\al_{1}}\de_1&-\frac{h}{\al_2}+\frac{h^2}{\al_2\widehat\al_2}\de_2&\frac{h}{\widehat\al_1}\de_1\\
\frac{\al_1}{\widehat\al_{1}}\de_1&-\frac{h}{\al_2}+\frac{h^2}{\al_2\widehat\al_2}\de_2&-\frac{h}{\al_1}+\frac{h^2}{\al_1\widehat\al_1}\de_1\\
-\frac{h}{\al_1}+\frac{h^2}{\al_1\widehat\al_1}\de_1&-\frac{h}{\al_2}+\frac{h^2}{\al_2\widehat\al_2}\de_2&\frac{\al_1}{\widehat\al_{1}}\de_1\\
\end{pmatrix}\cdot 1\\
&=h^2(h+\al_1). 
\end{align*}

If $I_w=s_1s_2s_1, u=s_1, v=s_1s_2$, then
\begin{align*}
\bbt^{s_1s_2s_1}_{s_1,s_1s_2}&=\widehat\al_1\widehat\al_2\widehat\al_{13}(\bft^{121}_{\{1\}, \{1,2\}}+\bft^{121}_{\{3\}, \{1,2\}})\\
&=\widehat\al_1\widehat\al_2\widehat\al_{13}\begin{pmatrix}
\frac{\al_1}{\widehat\al_{1}}\de_1&\frac{h}{\widehat\al_2}\de_2& -\frac{h}{\al_1}+\frac{h^2}{\al_1\widehat\al_1}\de_1\\
\frac{h}{\widehat\al_1}\de_1& \frac{h}{\widehat\al_2}\de_2&\frac{h}{\widehat\al_1}\de_1
\end{pmatrix}\cdot 1=h^2(h+\al_1).
\end{align*}

Similarly, for $v'=s_2s_1$, we have 
\begin{align*}\bbt^{s_1s_2s_1}_{s_1,s_2s_1}&=\widehat\al_{w_0}^2(\bft^{s_1s_2s_1}_{\{1\}, \{2,3\}}+\bft^{s_1s_2s_1}_{\{3\}, \{2,3\}})\\
&=\widehat\al_1\widehat\al_2\widehat\al_{13}\begin{pmatrix}
\frac{h}{\widehat\al_{1}}\de_1&\frac{h}{\widehat\al_2}\de_2& \frac{h}{\widehat\al_1}\de_1\\
-\frac{h}{\al_1}+\frac{h^2}{\al_1\widehat\al_1}\de_1& \frac{h}{\widehat\al_2}\de_2&\frac{\al_1}{\widehat\al_1}\de_1
\end{pmatrix}\cdot 1\\
&=h^3-h^2\widehat\al_{13}=h^2(\al_1+\al_2). 
\end{align*}
\end{ex}
\begin{ex}Consider the $A_3$ case. For $I_w=s_1s_2s_3s_1s_2, u=s_2s_3s_2, v=s_1s_2s_1$, with $\alpha_{ij}=\al_i+\cdots +\al_{j-1}$ for $1\le i<j\le 4$, we have
\begin{align*}
\bbt^{s_1s_2s_3s_1s_2}_{s_2s_3s_2, s_1s_2s_1}&=\widehat\al_{w_0}(\bft^{s_1s_2s_3s_1s_2}_{\{2,3,5\}, \{1,2,4\}}+\bft^{s_1s_2s_3s_1s_2}_{\{2,3,5\}, \{2,4,5\}})\\
&=\widehat\al_{w_0}\begin{pmatrix}
\frac{h}{\widehat\al_1}\de_1& \frac{\al_2}{\widehat\al_2} \de_2& \frac{h}{\widehat\al_3} \de_3 & \frac{h}{\widehat\al_1}\de_1&\frac{h}{\widehat\al_2}\de_2\\
-\frac{h}{\al_1}+\frac{h^2}{\al_1\widehat\al_1}\de_1& \frac{\al_2}{\widehat\al_2}\de_2&\frac{h}{\widehat\al_3} \de_3&\frac{h}{\widehat\al_1} \de_1& \frac{\al_2}{\widehat\al_2}\de_2
\end{pmatrix}\cdot 1\\
&=h^4\widehat\al_3(\al_1+\al_2)+h^3\widehat\al_3 (h\al_3+\al_1\al_2+\al_2^2+\al_2\al_3)\\
&=h^3\widehat\al_3(h+\al_2)(\al_1+\al_2+\al_3). 
\end{align*}
\begin{align*}
\bbt^{12312}_{232, 1}&=\widehat\al_{w_0}(\bft^{12312}_{\{2,3,5\}, \{1\}}+\bft^{12312}_{\{2,3,5\}, \{4\}})\\
&=h^5(\widehat\al_2+2\widehat\al_3).\\
\bbt^{12312}_{232, 2}&=\widehat\al_{w_2}(\bft^{12312}_{\{2,3,5\},\{2\}}+\bft^{12312}_{\{2,3,5\}, \{5\}})\\
&=h^4(3h^2+h\al_2+(\al_{24})\widehat\al_{14}). 
\end{align*}
\end{ex}
\begin{rem}\label{rem:Su}

In \cite[Theorem 1.1]{S19}, the authors find a formula for the structure constants of $\sigma^*_w\in (Q_W^a)^*$, where 
\[
\sigma_i=\frac{1+\al_i}{\al_i}\de_i-\frac{1}{\al_i}\in Q_W^a. 
\]
This is equal to our $-T_\al$ with $\hbar=-1$.
\end{rem}

\section{Structure constants for K-theoretic stable bases}
In this section, we give a formula of  the structure constants of the K-theory stable basis. Similar to our strategy in \textsection\ref{sec:cohstab}, we use the twisted group algebra method. This method was introduced by Su, Zhao and the second author in \cite{SZZ17}; we only recall the definitions below. 
Here we use $F=F_m$ and ${R=R^m = \bbZ[q^{1/2},q^{-1/2}]}$.

Let 
 $S^m=R^m[\La]$. We use the following notation in this section:
\[
x_{\pm\al}=1-e^{\mp\al}, \quad \hx_\al=1-qe^{-\al}, \quad \hx_w=\prod_{\al>0, w^{-1}\al<0}\hx_\al, \quad q_w=q^{\ell(w)}. 
\]
Let $Q^m=\Frac(S^m)$ and apply the twisted group algebra construction to obtain the module
 $$
 Q^m_W=Q^m\rtimes_{R^m} R^m[W].
 $$
  Define the operator $\tau_\al^-$ by 
\[
\tau_\al^-=\frac{q-1}{1-e^\al}+\frac{1-qe^{-\al}}{1-e^\al}\de_\al\in Q_W^m. 
\]
Observe that $\tau^-_\al$ is a special case of $Z_\al$ when $Q=Q^m$.

 A simple calculation shows that
$(\tau_\al^-)^2=(q-1)\tau^-_\al+q$, and that $\{\tau_\al\}$ satisfies the braid relations. 
It follows that the K-theoretic Demazure-Lusztig operator $\tau_w^-$, given by the product
$$
\tau_w^- = \tau_{\al_{i_1}}^- \tau_{\al_{i_2}}^-\cdots  \tau_{\al_{i_\ell}}^-,
$$
is independent of choice of reduced word $s_{i_1}s_{i_2}\cdots s_{i_\ell}$ for $w$. 
The set ${\{ \tau_w^-, w\in W\}}$ is a $Q^m$-basis of $Q_W^m$.

For each not-necessarily reduced sequence $I= (i_1,\dots, i_\ell)$,  let $\tau_I^-$ be the concatenation $\tau_{\al_{i_1}}^- \cdots \tau_{\al_{i_\ell}}^-$, and  define the structure constants $c^{\tau^-}_{I,w}\in R^m$ by the equations
\begin{equation}\label{tau}
\tau^-_I=\sum_{w\in W}c^{\tau^-}_{I,w}\tau_w^-.
\end{equation}

\begin{lem}\label{lem:Ktau} The coefficients $c^{\tau^-}_{I,w}\in R^m$ in \eqref{tau} satisfy the following:
\begin{enumerate}
\item For all $w\in W$ and sequences $I$, $c^{\tau^-}_{I,w}=0$ unless $w\le \wt\prod I$.
\item 
If $I$ is reduced, then
$$
c^{\tau^-}_{I,w}=\begin{cases}
0 &\mbox{if $w\neq \prod I$}\\
1 & \mbox{if $w= \prod I$}.
\end{cases}
$$

\end{enumerate}
\end{lem}
\begin{proof}Statement (1)  follows from the quadratic relation  $(\tau^-_\al)^2=(q-1)\tau_\al^-+q$. 

Statement (2) follows from the braid relations satisfied by the $\tau_\al^-$.
\end{proof} The analogous statement to Theorem~\ref{thm:Leibniz} is the following proposition.
\begin{prop}[K-Stable Leibniz Rule] If $I=(i_1,...,i_k)$, we have 
\[
\tau^-_I\cdot (pq)=\sum_{E,F\subset [k]}P^I_{E,F}(\tau_{I \vert_{E}}\cdot p) (\tau_{I \vert_{F}}\cdot q), \quad p,q\in Q.
\]
where $P^I_{E,F}=(B^{\tau^-}_1 B^{\tau^-}_2\cdots  B^{\tau^-}_k)\cdot 1$ with $B^{\tau^-}_j\in Q_W^m$ defined by
\[
B_j^{\tau^-}=\begin{cases}
\frac{1-e^{\al_{i_j}}}{1-qe^{-\al_{i_j}}}\de_{i_j}, &\text{ if }j\in E\cap F,\\
\frac{1-q}{1-qe^{-\al_{i_j}}}\de_{i_j}, &\text{ if }j\in  E \text{ or }F, \text{ but not both},\\
\frac{q-1}{1-qe^{-\al_{i_j}}}\tau^-_{\al_{i_j}}\de_{i_j}, &\text{ if }j\not\in E\cup F.\\
\end{cases}
\]

\end{prop}

Similar to \textsection\ref{sec:cohstab}, we take the dual $(Q_W^m)^*$, and $Q_W^m$ acts on $(Q_W^m)^*$ via the $\bullet$-action. Indeed, we have 
\[
(Q_W^m)^*\cong Q^m\otimes_{S^m}K_{\bbC^*\times T}(G/B)\cong Q^m\otimes_{S^m}K_{\bbC^*\times T}(T^*G/B). 
\]
\begin{defn}\cite[Definition 5.3, Theorem 5.4]{SZZ17}  The K-theoretic stable basis elements are defined by
\[
\stab^-_w=q_{w_0}q_w^{-1/2}(\tau^-_{w_0w})^{-1}\bullet (\prod_{\al>0}(1-e^{\al})f_{w_0})\in (Q_W^m)^*. 
\]
\end{defn}

Moreover, by \cite[Theorem 5.4, Theorem 6.5]{SZZ17}, we have 
\[
\stab^-_w=q_w^{1/2}\hx_{w_0}(\tau^-_w)^*. 
\]

The following theorem gives a formula for the structure constants of the K-theory stable basis:
\begin{thm} Let $\{\stab_w^- \vert \ w\in W\}$ denote the K-theory stable basis of $(Q_W^m)^*.$ Define coefficients $p^w_{u,v}\in Q^m$ by the equation
\[
\stab_u^-\cdot \stab_v^-=\sum_{w\ge u, w\ge v}p^w_{u,v}\stab_w^-.
\]
Then  
\[
p^w_{u,v}=q^{\frac{1}{2}(\ell(u)+\ell(v)-\ell(w))}\hx_{w_0}\sum P^{I_w}_{E,F}c^{\tau^-}_{I_w \vert_{E},u}c^{\tau^-}_{I_w \vert_{F}, v}, 
\]
where the sum is over all $E, F\subset[\ell(w)]$ such that $\widetilde\prod (I_{w} \vert_{E})\ge u$ and $\widetilde\prod (I_w \vert_{F})\ge v$, and coefficients $c^\tau_{I_w \vert_{F}, v}$ are given in Lemma \ref{lem:Ktau}
\end{thm}
\begin{proof}The proof follows a similar argument as that of Theorem \ref{thm:cohstab}. 
\end{proof}
\begin{rem}Due to the quadratic relation $(\tau_\al^-)^2=(q-1)\tau^-_\al+q$, it is difficult to express the sum in terms of formulas in Section~\ref{sec:cohK} and Section~\ref{sec:cohstab}. Indeed, this is also the reason why it is difficult to express the restriction formula of $\stab_w^-$ in \cite{SZZ17} in terms of an AJS-Billey-Graham-Willems type formula. 
\end{rem}

\section{The restriction formula}\label{se:restriction}

In this section we relate the structure constants of $Z_{I_w}^*$ with its restriction coefficients. This generalizes such relations in cohomology and K-theory due to Kostant and Kumar in  \cite[Proposition 4.32]{KK86} and \cite[Lemma 2.25]{KK90}. 

Let $Z_\al$ be given in Definition~\ref{def:LeibnizCoeff}. Following Lemma~\ref{coefficientsforinverting}, we obtain coefficients $b^Z_{u, I_w}\in Q$ using the defining relations
\[
\de_u=\sum_{w\in W}b^Z_{u,I_w}Z_{I_w},
\]
Then $Z_{I_v}^*=\sum_{u}b^Z_{u, I_v}f_u$, i.e.,  $Z_{I_v}^*(\de_u)=b^Z_{u, I_v}$. We call $b^Z_{u,I_v}$ the restriction coefficients of $Z_{I_v}^*$.
\begin{thm} \label{thm:res}For any $w\in W$, define the matrix $\fp_w^Z$ with $\fp^Z_{w}(u,v)=\bbz^{I_v}_{I_w, I_u}$, the matrix $\fb^Z$ with $\fb^Z(u,v)=b^Z_{v,I_u}$, and the matrix $\fb^Z_w$ with $\fb^Z_w(u,v)=\de_{u,v}b^Z_{u, I_w}$. Then 
\[
\fp_w^Z=\fb^Z \cdot \fb^Z_w \cdot(\fb^Z)^{-1}. 
\]
\end{thm}
\begin{proof}
We have 
\begin{eqnarray*}
(\fp^Z_w\cdot \fb^Z)(u,v)&=\sum_{z\in W}\fp^Z_w(u,z)\fb^Z(z,v)
&=\sum_{z\in W}\bbz^{I_z}_{I_w,I_u}b^Z_{v,I_z}\\
&=\sum_{z\in W}\bbz^{I_z}_{I_w, I_u}Z^*_{I_z}(\de_v)
&=(Z_{I_u}^*\cdot Z_{I_w}^*)(\de_v)\\
&=Z_{I_u}^*(\de_v)\cdot Z_{I_w}^*(\de_v)
&=b^Z_{v, I_u}b^Z_{v, I_w}\\
&=\sum_{z\in W} b^Z_{z,I_u} \de_{z,v}b^Z_{z,I_w}
&=\sum_{z\in W}\fb^Z(u,z)\fb^Z_w(z,v)\\
&=(\fb^Z\cdot \fb^Z_w)(u,v). 
\end{eqnarray*}
\end{proof}
\begin{cor}\label{cor:res}For any $v,w\in W$, we have
\[
\bbz^{I_v}_{I_w,I_v}=b^Z_{v,I_w}.
\]
In particular, $\bbz^{I_v}_{I_w,I_v}$ does not depend on the choice of $I_v$. 
\end{cor}
\begin{proof}Denote $Z_{I_w}=\sum_{v\le w}a^Z_{I_w,v}\de_v$. Then the matrix $\fa^Z$ with $\fa^Z(u,v)=a^Z_{I_v,u}$ is the inverse of $\fb^Z$. Theorem \ref{thm:res} implies that 
\begin{align*}
\bbz^{I_v}_{I_w, I_v}&=\fp^Z_w(v,v)\\
&=\sum_{z_1,z_2\in W}\fb^Z(v,z_1)\fb^Z_w(z_1,z_2)\fa^Z(z_2,v)\\
&=\sum_{z_1\ge v,z_2\le v}b^Z_{{z_1}, I_v}\de_{z_1,z_2}b^Z_{z_1,I_w}a^Z_{I_v,{z_2}}\\
&=\sum_{v\le z_1\le v}b^Z_{z_1,I_v}b^Z_{z_1,I_w}a^Z_{I_v, z_1}\\
&=b^Z_{v,v}b^Z_{v,I_w}a^Z_{v,v}=b^Z_{v,I_w}=Z^*_{I_w}(\de_v). 
\end{align*}
\end{proof}

\begin{rem} Corollary \ref{cor:res} can also be proved with the following argument: since $Z_{I_w}^*(\de_u)=b^Z_{u,I_w}=0$ unless $u\ge w$, 
\begin{align*}
(Z_{I_u}^*\cdot Z_{I_v}^*)(\de_u)&=(\sum_{w\ge u, w\ge v}\bbz^{I_w}_{I_u, I_v}Z_{I_w}^*)(\de_u)\\
&=\bbz^{I_u}_{I_u, I_v}Z_{I_u}^*(\de_u)=\bbz^{I_u}_{I_u, I_v}b^Z_{u, I_u}.
\end{align*}
On the other hand, 
\[(Z_{I_u}^*\cdot Z_{I_v}^*)(\de_u)=Z_{I_u}^*(\de_u)Z_{I_v}^*(\de_u)=b^Z_{u, I_u}b^Z_{u, I_v}.\] Therefore, $\bbz^{I_u}_{I_u,I_v}=b^Z_{u, I_v}$. 
\end{rem}

\begin{rem} As mentioned in \cite{GK19}, specializing Corollary \ref{cor:res} and Examples \ref{ex:FaSch} and \ref{ex:FmSch} to singular cohomology or K-theory, and $Z_\al$ to the  $X_\al$ and $Y_\al$-operators, one recovers the AJS/Billey formula and Graham-Willems formula of restriction coefficients of Schubert classes, which are obtained by using root polynomials.
\end{rem}

\begin{ex}Consider the $A_2$-case with $w=s_1, v=s_1s_2s_1$. We compute $b^X_{v,w}=X_{I_w}^*(\de_v)$. For $\bfx^{I_v}_{[3], E}$, we only need to consider the following three:
\begin{align*}
\bfx^{I_v}_{[3], \{1\}}=-x_1, & \quad \bfx^{I_v}_{[3], \{3\}}=-x_{2},\quad \bfx^{I_v}_{[3], \{1,3\}}=x_{1}x_2.
\end{align*}
On the other hand, $c^X_{I_v \vert_E,I_w}=1$ when $E=\{1\}, \{3\}$, and $X_1X_1=\ka_1X_1$. So $c^X_{I_v \vert_{\{1,3\}}, s_1}=\ka_1$. Therefore, 
\begin{align*}
b^X_{w,I_v}=-x_1-x_2+\ka_1x_1x_2.
\end{align*}
In particular, if $F=F_a$, then $b^X_{w,I_v}=-x_1-x_2$, and if $F=F_{m}$, then $b^X_{w,I_v}=-x_1-x_2+x_1x_2=-x_{1+2}$, with $x_{1+2}=x_{\al_{1}+\al_2}$. 
\end{ex}

\begin{ex}Let $w=s_1s_2, v=s_1s_2s_3s_1s_2$. Let us compute $b^X_{v,I_w}=X_{I_w}^*(\de_v)$. We write $X_{ijk\cdots}$ for $X_{i}X_jX_k\cdots$,  $x_{\pm i\pm j}=x_{\pm \al_i\pm \al_j}$ and $\kappa_{\pm i, \pm j}=\kappa_{\pm \al_i, \pm \al_j}$.  To compute $\bfx^{I_v}_{[6], E}$, we only need to consider
\begin{align*}
\bfx^{I_v}_{[6], \{1,2\}}=x_1x_{1+2},& \quad
\bfx^{I_v}_{[6], \{1,5\}}=x_1x_{2+3},\\
\bfx^{I_v}_{[6], \{4,5\}}=x_2x_{2+3},&\quad
\bfx^{I_v}_{[6], \{1,2,5\}}=-x_1x_{1+2}x_{2+3},\\
\bfx^{I_v}_{[6], \{1,4,5\}}=-x_1x_2x_{2+3},&\quad
\bfx^{I_v}_{[6], \{1,2,4,5}\}=x_1x_{1+2}x_2x_{2+3}.
\end{align*}
On the other hand, $c^X_{I_v \vert_{E}, I_w}=1$ when $E=\{1,2\}, \{1,5\}, \{4,5\}$. Concerning $X_{I_w \vert_{\{1,2,5\}}}=X_{122}$, since 
\[
X_1X_2X_2=X_1\ka_2X_2=s_1(\ka_2)X_{12}+\Delta_1(\ka_2)X_{2}=\ka_{1+2}X_{12}+\Delta_1(\ka_2)X_2,
\]
so 
\[c^X_{I_v \vert_{\{1,2,5\}}, I_w}=\kappa_{1+2}.\]
 For $X_{I_w \vert_{\{1,4,5\}}}=X_{112}$, from $X_1X_1X_2=\ka_1X_{1}X_2$, we get \[c^X_{I_v \vert_{\{1,4,5\}}, I_w}=\kappa_1.\]
 Lastly, for $X_{I_w \vert_{\{1,2,4,5\}}}=X_{1122}$, from Lemma \ref{lem:1} we know 
\begin{align*}
X_{1212}&=X_1(X_{121}+\ka_{12}X_1-\ka_{21}X_2)\\
&=\ka_1X_{121}+X_1\ka_{12}X_1-X_1\ka_{21}X_2\\
&=\ka_1X_{121}+s_1(\ka_{12})X_1^2+\Delta_1(\ka_{12})X_1-s_1(\ka_{21})X_{12}-\Delta_1(\ka_{21})X_2\\
&=\ka_1X_{121}+\ka_{-1, 1+2}\ka_1X_1+\Delta_1(\ka_1)X_1-\ka_{1+2,-1}X_{12}-\Delta_1(\ka_{21})X_2, 
\end{align*}
so 
\[c^X_{I_v \vert_{\{1,2,4,5\}}, I_w}=s_1(\ka_{21})=-\ka_{1+2,-1}.\]
Therefore, 
\begin{align*}
&b^X_{s_1s_2,s_1s_2s_3s_1s_2}\\
&=\bfx^{I_v}_{[6], \{1,2\}}+\bfx^{I_v}_{[6], \{1,5\}}+\bfx^{I_v}_{[6], \{4,5\}}\\
&+
\bfx^{I_v}_{[6], \{1,2,5\}}\ka_{1+2}+\bfx^{I_v}_{[6], \{1,4,5\}}\ka_1+\bfx^{I_v}_{[6], \{1,2,4,5\}}(-\ka_{1+2, -1})\\
&=x_1x_{1+2}+x_1x_{2+3}+x_2x_{2+3}-x_1x_{1+2}x_{2+3}\ka_{1+2}-x_1x_2x_{2+3}\ka_1-x_1x_{1+2}x_{2+3}\ka_{1+2,-1}\\
&=x_1x_{1+2}+x_1x_{2+3}+x_2x_{2+3}-x_{2+3}(x_1+x_2+\frac{x_1}{x_{-1}}x_{1+2}).
\end{align*}
In particular, if $F=F_a$, then 
\[b^X_{s_1s_2,s_1s_2s_3s_1s_2}=\al_1(\al_{1}+\al_{2})+\al_1(\al_{2}+\al_{3})+\al_2(\al_{2}+\al_{3}).\]
 If $F=F_m$, then 
\[b^X_{s_1s_2,s_1s_2s_3s_1s_2}=x_1x_{1+2}+x_1x_{2+3}+x_2x_{2+3}-x_1x_{2+3}(x_{1+2}+x_2).\]
 These agree with the result computed by using root polynomials. 
\end{ex}

\newcommand{\arxiv}[1]
{\texttt{\href{http://arxiv.org/abs/#1}{arXiv:#1}}}

\bibliographystyle{halpha}

\end{document}